\documentclass[10pt]{article}

\usepackage{amssymb,latexsym,amsmath,enumerate,verbatim,amsfonts,amsthm,cite}
\usepackage{color} 
\usepackage[active]{srcltx}

\newtheorem{lemma}{Lemma}
\newtheorem{definition}{Definition}

\newtheorem{theorem}{Theorem}
\newtheorem{corollary}{Corollary}
\newtheorem{proposition}{Proposition}
\newtheorem{remark}{Remark}
\newtheorem{example}{Example}

\def\R{{\rm I\!R}}
\def\D{{\frak D}}

\def\hat{\widehat}

\def\tilde{\widetilde}
\def\Bar{\overline}
\def\ra{\rangle}
\def\la{\langle}

\def\oR{\overline\mathbb{R}}

\def\st{\stackrel}
\def\oR{\Bar{\R}}

\def\Argmin{\mathop{\rm Arg\,min}}
\def\argmin{\mathop{\rm arg\,min}}

\def\disp{\displaystyle}

\def\tto{\;{\lower 1pt \hbox{$\rightarrow$}}\kern -10pt
\hbox{\raise 2pt \hbox{$\rightarrow$}}\;}

\title{\sf Douglas-Rachford splitting for nonconvex optimization with application to nonconvex feasibility problems}

\author{
Guoyin Li \thanks{Department of Applied
Mathematics, University of New South Wales, Sydney 2052, Australia.
E-mail: {g.li@unsw.edu.au}. This author was partially supported by a research grant from Australian Research Council.}
\and Ting Kei Pong \thanks{Department of Applied Mathematics, the Hong Kong Polytechnic University, Hong Kong.
This author was supported partly by a research grant from Hong Kong Polytechnic University. E-mail: {tk.pong@polyu.edu.hk}.}
}

\date{Second Revised version: August 13, 2015}

\begin{document}
\maketitle

\begin{abstract}
We adapt the Douglas-Rachford (DR) splitting method to solve nonconvex feasibility problems by studying this method for a class of nonconvex optimization problem. While the convergence properties of the method for convex problems have been well studied, far less is known in the nonconvex setting. In this paper, for the direct adaptation of the method to minimize the sum of a proper closed function $g$ and a smooth function $f$ with a Lipschitz continuous gradient, we show that if the step-size parameter is smaller than a computable threshold and the sequence generated has a cluster point, then it gives a stationary point of the optimization problem. Convergence of the whole sequence and a local convergence rate are also established under the additional assumption that $f$ and $g$ are semi-algebraic. We also give simple sufficient conditions guaranteeing the boundedness of the sequence generated. We then apply our nonconvex DR splitting method to finding a point in the intersection of a closed convex set $C$ and a general closed set $D$ by minimizing the squared distance to $C$ subject to $D$. We show that if either set is bounded and the step-size parameter is smaller than a computable threshold, then the sequence generated from the DR splitting method is actually bounded. Consequently, the sequence generated will have cluster points that are stationary for an optimization problem, and the whole sequence is convergent under an additional assumption that $C$ and $D$ are semi-algebraic. We achieve these results based on a new merit function constructed particularly for the DR splitting method. Our preliminary numerical results indicate that our DR splitting method usually outperforms the alternating projection method in finding a sparse solution of a linear system, in terms of both the solution quality and the number of iterations taken.
\end{abstract}

\section{Introduction}
Many problems in diverse areas of mathematics, engineering and physics aim at finding a point in the intersection
of two closed sets.  This problem is often called the feasibility problem. Many practical optimization problems and reconstruction problems can be cast in this framework. We refer the readers to the comprehensive survey \cite{Bauschke_SIAM_review} and the recent
monograph \cite{BauCom11} for more details.

The Douglas-Rachford (DR) splitting method is an important and powerful algorithm that can be applied to solving problems with competing
structures, such as finding a point in the intersection of two closed convex sets (feasibility problem), or, more generally, minimizing the sum of two proper closed convex functions.
The latter problem is more general because the feasibility problem can be viewed as a minimization problem that minimizes the sum of the indicator functions of the two sets.
In typical applications,
the projection onto each of the constituent sets is simple to compute in the feasibility problem,
and the so-called proximal operator of each of the constituent functions is also easy to compute in the case of minimizing the sum of two functions.
Since these simple operations are usually the main computational parts of the DR splitting method, the method can be implemented efficiently in practice.

The DR splitting method aims at finding a point in the intersection of two closed sets in a Hilbert space, and was originally introduced in \cite{PDE} to solve nonlinear heat flow problems.
Later, Lions and Mercier \cite{Lions_Mercier} showed that the DR splitting method converges for two closed convex sets with nonempty intersection.
This scheme was examined further in \cite{Eckstein_Bertsekas} again in the convex setting, and its relationship with another popular method,
the proximal point algorithm, was revealed and explained therein. Recently, the DR splitting method has also been applied to various optimization problems that arise from signal processing and other applications, where the objective is the sum of two proper closed convex functions; see, for example, \cite{ComPes07,GanRechYa11,He,PaStBe14}. We refer the readers to the recent exposition \cite{BauCom11} and references therein for a discussion about convergence in the convex setting.

While the behavior of the DR splitting method has been moderately understood in the convex cases,
the theoretical justification is far from complete when the method is used in the nonconvex setting.
Nonetheless, the DR splitting method has been applied very successfully to various important problems where the underlying sets are not necessarily convex \cite{Jon_AZNIAM,Jon_JOTA}. This naturally motivates the following research direction:
\begin{quote}
\emph{Understand the DR splitting method when applied to possibly nonconvex sets.}
\end{quote}
As commented in \cite{Luke1}, the DR splitting method is notoriously
difficult to analyze compared with other projection type methods such as the alternating projection method \cite{Bauschke_SIAM_review,Jon1993,Lewis,Li}.
Despite its difficulty, there has been some important recent progress towards understanding the behavior of the DR splitting method in the nonconvex setting.
For example, it was shown in \cite{Luke1} that the DR splitting method exhibits local linear convergence for an affine set and a super-regular set (an extension of convexity which emphasizes local features), under suitable regularity conditions. Very recently, Phan improved the result in \cite{Luke1} and obtained local linear convergence results of DR splitting method for two super-regular sets \cite{Phan14}.  There are also recent advances in dealing with specific structures such as the case where the two sets are finite union of convex sets \cite{Bauschke}, and the sparse feasibility problem where one seeks a sparse solution of a linear system \cite{Luke2}. On the other hand, in spite of the various local convergence results, global convergence of the method was only established in \cite{Jon1} for finding the intersection of a line and a circle.

In this paper, we approach the above basic problem from a new perspective. Recall that the alternating projection method for finding a point in the intersection of a closed convex set $C$ and a closed set $D$ can be interpreted as an application of the proximal gradient algorithm to the optimization problem
\begin{equation}\label{introdc}
\min_{x\in D} \ \frac12 d_C^2(x)
\end{equation}
with step-length equals 1, where $d_C(x)$ is the distance from $x$ to $C$ and $x\mapsto d_C^2(x)$ is smooth since $C$ is convex. Motivated by this, we adapt the DR splitting method to solve the above optimization problem instead, which is conceivably easier to analyze due to the smooth objective. Notice that the feasibility problem is solved when the globally optimal value of \eqref{introdc} is zero.

We note that this approach is different from the common approach in the literature (see, for example, \cite{Jon1,Jon_AZNIAM,Jon_JOTA,Luke2,Bauschke}) where the DR splitting method is applied to minimizing the sum of indicator functions of the two sets. On the other hand, an approach similar to ours was considered recently in \cite{Luke08}, which studied a more general framework of algorithms; however, only local convergence of in the case when $D$ is prox-regular was established there. In our work, we aim at analyzing {\em both} global and local convergence.

In our analysis, we start with a more general setting: minimizing the sum of a smooth function $f$ with Lipschitz continuous gradient, and a proper closed function $g$. We show that, if the step-size parameter is smaller than a computable threshold and the sequence generated from the DR splitting method has a cluster point, then it gives a stationary point of the optimization problem $\min_{x\in\R^n}\{f(x) + g(x)\}$. Moreover, under the additional assumption that $f$ and $g$ are semi-algebraic, we show convergence of the whole sequence and give a local convergence rate. In addition, we also give simple sufficient conditions guaranteeing the boundedness of the sequence generated, and hence the existence of cluster points. Our analysis relies heavily on the so-called Douglas-Rachford merit function (see Definition~\ref{def:DRm}) we introduce, which is non-increasing along the sequence generated by the DR splitting method when the step-size parameter is chosen small enough.

We then apply our nonconvex DR splitting method to minimizing \eqref{introdc}, whose objective is smooth with a Lipschitz continuous gradient.
When the step-size parameter is smaller than a computable threshold and either set is compact, we show that the sequence generated from the DR splitting method is bounded.
Thus, cluster points exist and they are stationary for \eqref{introdc}. Furthermore, if $C$ and $D$ are in addition semi-algebraic, we show that the whole sequence is convergent.
Finally, we perform numerical experiments to compare our method against the alternating projection method on finding a sparse solution of a linear system. Our preliminary numerical results show that the DR splitting method usually outperforms the alternating projection method, in terms of both the number of iterations taken and the solution quality.

The rest of the paper is organized as follows. We present notation and preliminary materials in Section~\ref{sec2}. The DR splitting method applied to minimizing the sum of a smooth function $f$ with Lipschitz continuous gradient and a proper closed function $g$ is analyzed in Section~\ref{sec3}, while its application to a nonconvex feasibility problem is discussed in Section~\ref{sec4}. Numerical simulations are presented in Section~\ref{sec5}. In Section~\ref{sec6}, we present some concluding remarks.

\section{Notation and preliminaries}\label{sec2}

We use $\R^n$ to denote the $n$-dimensional Euclidean space, $\langle\cdot,\cdot\rangle$
to denote the inner product and $\|\cdot\|$ to denote the norm induced from the inner product.
For an extended-real-valued function $f$, the domain of $f$ is defined as ${\rm dom}f:=\{x \in \R^n: f(x)<+\infty\}$. The function is called proper if ${\rm dom}f\neq \emptyset$ and it is never $-\infty$. The function is called closed if it is lower semicontinuous.
For a proper function $f:\R^n \to\oR:=(-\infty,\infty]$, let $z\st{f}{\to}x$ denote $z\to x$ and $f(z)\to f(x)$. Our basic {\em subdifferential} of $f$ at $x\in\mathrm{dom}\,f$ (known also as the limiting subdifferential) is defined by
\begin{equation}\label{ls}
\partial f(x):=\left\{v\in\R^n :\;\exists x^t\st{f}{\to}x,\;v^t\to v\;\mbox{ with }\disp\liminf_{z\to x^t}\frac{f(z)-f(x^t)-\la v^t,z-x^t\ra}{\|z-x^t\|}\ge 0\mbox{ for each }t\right\}.
\end{equation}
The above definition gives immediately the following robustness property:
\begin{equation}\label{outersemi}
  \left\{v\in \R^n:\; \exists x^t \st{f}{\to}x,\; v^t\to v\;, v^t\in \partial f(x^t)\right\} \subseteq \partial f(x).
\end{equation}
We also use the notation ${\rm dom}\,\partial f := \{x\in \R^n:\; \partial f(x)\neq \emptyset\}$.
The subdifferential \eqref{ls} reduces to the derivative of $f$ denoted by $\nabla f$ if $f$ is continuously differentiable. On the other hand, if $f$ is convex, the subdifferential \eqref{ls} reduces to the classical subdifferential in convex analysis (see, for example, \cite[Proposition~8.12]{Rock98}), i.e.,
\begin{eqnarray*}
\partial f(x)=\left\{v\in \R^n :\;\langle v,z-x\rangle\le f(z)-f(x)\ \ \forall \ z\in\R^n\right\}.
\end{eqnarray*}
For a function $f$ with several groups
of variables, we write $\partial_x f$ (resp., $\nabla_x f$) for the subdifferential (resp., derivative) of $f$ with respect to the variable $x$. We say that a function $f$ is coercive if $\liminf_{\|x\|\to \infty}f(x) = \infty$. Finally, we say a function $f$ is a strongly convex function with modulus $\omega>0$ if $f-\frac{\omega}{2}\|\cdot\|^2$ is a convex function.

For a closed set $S \subseteq \R^n$, its indicator function $\delta_S$ is defined by
\[
\delta_S(x)=\begin{cases}
0 & \mbox{ if }  x \in S, \\
+\infty & \mbox{ if }  x \notin S.
\end{cases}
\]
Moreover, the (limiting) normal cone of $S$ at $x \in S$ is given by \begin{equation}\label{eq:normal_cone}N_S(x)=\partial \delta_S(x). \end{equation}
Furthermore, we use ${\rm dist}(x,S)$ or $d_S(x)$ to denote the distance from $x$ to $S$, i.e., $\inf_{y\in S}\|x - y\|$. If a set $S$ is closed and convex, we use
$P_S(x)$ to denote the projection of $x$ onto $S$.

A semi-algebraic set $S\subseteq \R^n$
is a finite union of sets of the form
\[\{x \in \R^n: h_1(x) = \cdots = h_k(x) = 0,g_1(x) < 0,\ldots,g_l(x) < 0\},
\]
where $g_1,\ldots,g_l$ and $h_1,\ldots, h_k$ are real polynomials. 
A function
$F:\R^n \rightarrow \R$ is semi-algebraic if the set $\{(x,F(x)) \in \R^{n+1}:\; x\in \R^n\}$ is semi-algebraic. Semi-algebraic sets and semi-algebraic functions
can be easily identified and cover lots of possibly nonsmooth and nonconvex functions that arise in real world applications \cite{AtBoReSo10,AtBoSv13,BolDanLew07}.

We will also make use of the following Kurdyka-{\L}ojasiewicz (KL) property that holds in particular for semi-algebraic functions.

\begin{definition}\label{def:KL} {\bf (KL property \& KL function)}
  We say that a proper function $h$ has the Kurdyka-{\L}ojasiewicz (KL) property at $\hat x\in {\rm dom}\,\partial h$ if there exist a neighborhood $\cal V$ of $\hat x$, $\nu\in (0,\infty]$ and a continuous concave function $\psi:[0,\nu)\rightarrow \R_+$ such that:
  \begin{enumerate}[{\rm (i)}]
    \item $\psi(0) = 0$ and $\psi$ is continuously differentiable on $(0,\nu)$ with $\psi' > 0$;
    \item for all $x\in {\cal V}$ with $h(\hat x)< h(x) < h(\hat x) + \nu$, it holds that
    \begin{equation*}
      \psi'(h(x) - h(\hat x))\,{\rm dist}(0,\partial h(x))\ge 1.
    \end{equation*}
  \end{enumerate}
  A proper closed function $h$ satisfying the KL property at all points in ${\rm dom}\,\partial h$ is called a KL function.
\end{definition}

It is known from
\cite[Section~4.3]{AtBoReSo10} that a proper closed semi-algebraic function always satisfies the KL property. Moreover, in this case, the KL property is satisfied with a specific form; see also \cite[Corollary~16]{BolDanLewShi07} and \cite[Section~2]{BolDanLew07} for further discussions.
\begin{proposition}{\bf (KL inequality in the semi-algebraic cases)}\label{Prop1}
Let $h$ be a proper closed semi-algebraic function on $\R^n$. Then, $h$ satisfies the KL property at all points in ${\rm dom}\,\partial h$ with $\psi(s) = cs^{1-\theta}$ for some $\theta\in [0,1)$ and $c>0$.
\end{proposition}

\section{Douglas-Rachford splitting for structured optimization}\label{sec3}

In this section, we consider the following structured optimization problem:
  \begin{equation}\label{P1}
    \min_u \ f(u) + g(u),
  \end{equation}
  where $f$ has a Lipschitz continuous gradient whose Lipschitz continuity modulus is bounded by $L$, and $g$ is a proper closed function \footnote{We note that
  the assumption where $f$ has a Lipschitz continuous gradient
is commonly used in the literature of first-order methods; see for example, \cite[Section~5]{AtBoSv13}.}.
  In addition, we let $l\in \R$ be such that $f + \frac{l}{2}\|\cdot\|^2$ is convex. Notice that such an $l$ always exists: in particular, one can always take $l = L$. Finally, for any given parameter $\gamma>0$, which will be referred to as
  a step-size parameter throughout this paper, we assume that the proximal mapping of $\gamma g$, is well defined and easy to compute, in the sense that it is simple to find a minimizer of the following problem for each given $z$, and that such a minimizer exists:
  \begin{equation}\label{proxg}
  \min_u \ \  \gamma g(u) + \frac12 \|u - z\|^2.
  \end{equation}

  Problems in the form of \eqref{P1} arise naturally in many engineering and machine learning applications. Specifically, many sparse learning problems take the form of \eqref{P1} where $f$ is a loss function and $g$ is a regularizer with \eqref{proxg} easy to compute; see, for example, \cite{GHLYZ13} for the use of a difference-of-convex function as a regularizer, and \cite{ZengLinWangXu14} for the case where $g(x) = \sum_{i=1}^n |x_i|^\frac12$.  Below, we consider a direct adaptation of the DR splitting method to solve \eqref{P1}.\\ \

\fbox{\parbox{5.7 in}{
\begin{description}
\item {\bf Douglas-Rachford splitting method}

\item[Step 0.] Input an initial point $x^0$ and a step-size parameter $\gamma > 0$.

\item[Step 1.] Set
\begin{equation}\label{scheme}
\left\{
\begin{split}
&y^{t+1}\in\Argmin_{y} \left\{f(y) + \frac{1}{2\gamma}\|y - x^t\|^2\right\}, \\
&z^{t+1}\in\Argmin_{z} \left\{g(z) + \frac{1}{2\gamma}\|2y^{t+1} - x^t - z\|^2\right\},\\
&x^{t+1}=x^t+(z^{t+1}- y^{t+1}).
\end{split}
\right.
\end{equation}

\item[Step 2.] If a termination criterion is not met, go to Step 1.
\end{description}
}}
\vspace{.1 in}

Using the optimality conditions and the subdifferential calculus rule \cite[Exercise~8.8]{Rock98}, we see from the $y$ and $z$-updates in \eqref{scheme} that
\begin{equation}\label{fyupdate}
  \begin{split}
    0 & = \nabla f(y^{t+1}) + \frac{1}{\gamma}(y^{t+1} - x^t),\\
    0 & \in \partial g(z^{t+1}) + \frac{1}{\gamma}(z^{t+1} - y^{t+1}) - \frac{1}{\gamma}(y^{t+1} - x^t).
  \end{split}
\end{equation}
Hence, we have for all $t\ge 1$ that
\begin{equation}\label{eq:stationary}
  0 \in  \nabla f(y^t) + \partial g(z^t) + \frac{1}{\gamma}(z^t - y^t).
\end{equation}
Thus, if
\begin{equation}\label{cond1}
  \lim_{t\rightarrow \infty}\|x^{t+1} - x^t\| = \lim_{t\rightarrow \infty}\|z^{t+1} - y^{t+1}\| = 0,
\end{equation}
and if we have for a cluster point $(y^*,z^*,x^*)$ of $\{(y^t,z^t,x^t)\}$ with a convergent subsequence $\lim_{j\rightarrow \infty}(y^{t_j},z^{t_j},x^{t_j})=(y^*,z^*,x^*)$ that
\begin{equation}\label{cond2}
  \lim_{j\rightarrow \infty}g(z^{t_j})= g(z^*),
\end{equation}
then passing to the limit in \eqref{eq:stationary} along the subsequence and using \eqref{outersemi}, it is not hard to see that $(y^*,z^*)$ gives a stationary point of \eqref{P1}, in the sense that $y^* = z^*$ and
\[
0\in \nabla f(z^*) + \partial g(z^*).
\]

In the next theorem, we establish convergence of the DR splitting method on \eqref{P1} by showing that \eqref{cond1} and \eqref{cond2} hold. The proof of this convergence result heavily relies on the following definition of the Douglas-Rachford merit function.

\begin{definition}{\bf (DR merit function)}\label{def:DRm}
Let $\gamma >0$. The Douglas-Rachford merit function is defined by
  \begin{equation}\label{Ddef}
    \D_\gamma(y,z,x) := f(y) + g(z) - \frac1{2\gamma}\|y - z\|^2 + \frac1\gamma \langle x-y,z-y\rangle.
  \end{equation}
\end{definition}
This definition was motivated by the so-called Douglas-Rachford envelope considered in \cite[Eq.~35]{PaStBe14} in the convex case (that is, when $f$ and $g$ are both convex). Moreover, we see that $\D_\gamma$ can be alternatively written as
  \begin{equation}\label{Drelation}
  \begin{split}
    \D_\gamma(y,z,x) & = f(y) + g(z) + \frac1{2\gamma}\|2y - z - x\|^2 - \frac{1}{2\gamma}\|x-y\|^2 - \frac1\gamma \|y - z\|^2\\
    & = f(y) + g(z) + \frac1{2\gamma}(\|x - y\|^2 - \|x - z\|^2)\end{split}
  \end{equation}
  where the first relation follows by applying the elementary relation $\langle u,v\rangle = \frac12 (\|u + v\|^2 - \|u\|^2 - \|v\|^2)$ in \eqref{Ddef} with $u = x-y$ and $v = z -y$, while the second relation follows by completing the squares in \eqref{Ddef}.

\begin{theorem}\label{thm1} {\bf (Global subsequential convergence)}
  Suppose that the parameter $\gamma > 0$ is chosen so that
  \begin{equation}\label{gammacond}
  (1 + \gamma L)^2 + \frac{5\gamma l}2 - \frac{3}{2} < 0.
  \end{equation}
  Then $\{\D_\gamma(y^t,z^t,x^t)\}_{t\ge 1}$ is nonincreasing.

  Moreover, if a cluster point of the sequence $\{(y^t,z^t,x^t)\}$ exists, then \eqref{cond1} holds. Furthermore, for any cluster point $(y^*,z^*,x^*)$, we have $z^* = y^*$, and
  \[
  0\in \nabla f(z^*) + \partial g(z^*).
  \]
\end{theorem}
\begin{remark}
  Notice that $\lim_{\gamma\downarrow 0}[(1 + \gamma L)^2 + \frac{5\gamma l}2 - \frac{3}{2}] = -\frac12 < 0$. Thus, given $l\in \R$ and $L > 0$, the condition \eqref{gammacond} will be satisfied for any sufficiently small $\gamma > 0$. Moreover, from the definition of $l$, we must have $l\in [-L,\infty)$. Hence, from the quadratic formula and some simple arithmetics, it is not hard to see that the $\gamma$ chosen as in \eqref{gammacond} has to satisfy
  \begin{equation}\label{gammaupperbound}
  \gamma < \frac{-(2.5 l + 2L) + \sqrt{(2.5 l + 2L)^2 + 2L^2}}{2L^2}\le \frac{1}{L},
  \end{equation}
  since the maximum of the fraction in the middle is achieved at $l = -L$ for each fixed $L$.
  We also comment on the $y$ and $z$-updates of the DR splitting method. Note that the $z$-update involves a computation of the proximal mapping of $\gamma g$, which is simple by assumption. On the other hand, from the choice of $\gamma$ in Theorem~\ref{thm1}, we have $l < \frac{3}{5\gamma} < \frac1\gamma$. This together with the assumption $f + \frac{l}{2}\|\cdot\|^2$ is convex shows that the objective function in the unconstrained smooth minimization problem for the $y$-update is a strongly convex function with  modulus $\frac{1}{\gamma}-l>0$.
\end{remark}
\begin{proof}
  We first study the behavior of $\D_\gamma$ along the sequence generated from the DR splitting method. First of all, notice from \eqref{Ddef} that
  \begin{equation}\label{D1ineq}
    \D_\gamma(y^{t+1},z^{t+1},x^{t+1}) - \D_\gamma(y^{t+1},z^{t+1},x^t) = \frac1\gamma\langle x^{t+1}-x^t,z^{t+1} - y^{t+1}\rangle = \frac1\gamma\|x^{t+1} - x^t\|^2,
  \end{equation}
  where the last equality follows from the definition of $x$-update.
  Next, using the first relation in \eqref{Drelation}, we obtain that
  \begin{equation}\label{D2ineq}
    \begin{split}
      &\D_\gamma(y^{t+1},z^{t+1},x^t) - \D_\gamma(y^{t+1},z^t,x^t) \\
      & = g(z^{t+1}) + \frac1{2\gamma}\|2y^{t+1} - z^{t+1} - x^t\|^2 - \frac1\gamma \|y^{t+1} - z^{t+1}\|^2 \\
      &\ \ \ \ - g(z^t) - \frac1{2\gamma}\|2y^{t+1} - z^t - x^t\|^2 + \frac1\gamma \|y^{t+1} - z^t\|^2\\
      & \le \frac1\gamma (\|y^{t+1} - z^t\|^2 - \|y^{t+1} - z^{t+1}\|^2) = \frac1\gamma (\|y^{t+1} - z^t\|^2 - \|x^{t+1} - x^t\|^2),
    \end{split}
  \end{equation}
  where the inequality follows from the definition of $z^{t+1}$ as a minimizer, and the
  last equality follows from the definition of $x^{t+1}$. Next, notice from the first relation in \eqref{fyupdate} that
  \[
  \frac{1}{\gamma}(x^t - y^{t+1}) + ly^{t+1} = \nabla \left(f + \frac{l}{2}\|\cdot\|^2\right)(y^{t+1}).
  \]
  Since $f + \frac{l}{2}\|\cdot\|^2$ is convex by assumption, using the monotonicity of the gradient of a convex function, we see that for all $t\ge 1$, we have
  \begin{equation*}
    \begin{split}
      &\left\langle \left(\frac{1}{\gamma}(x^t - y^{t+1}) + ly^{t+1}\right) - \left(\frac{1}{\gamma}(x^{t-1} - y^t) + ly^t\right),y^{t+1} - y^t\right\rangle \ge 0 \\
      &\Longrightarrow
      \langle x^{t} -x^{t-1},y^{t+1} - y^t\rangle \ge (1 - \gamma l)\|y^{t+1} - y^t\|^2.
    \end{split}
  \end{equation*}
  Hence, we see further that
  \begin{equation}\label{bdy}
    \begin{split}
      &\|y^{t+1} - z^t\|^2 = \|y^{t+1} - y^t + y^t - z^t\|^2 = \|y^{t+1} - y^t - (x^t - x^{t-1})\|^2\\
      & = \|y^{t+1} - y^t\|^2 - 2\langle y^{t+1} - y^t,x^t - x^{t-1}\rangle + \|x^t - x^{t-1}\|^2\\
      & \le (-1+2\gamma l)\|y^{t+1} - y^t\|^2 +  \|x^t - x^{t-1}\|^2,
    \end{split}
  \end{equation}
  where we made use of the definition of $x^t$ for the second equality. Plugging \eqref{bdy} into \eqref{D2ineq}, we obtain that whenever $t\ge 1$,
  \begin{equation}\label{D2ineq2}
    \D_\gamma(y^{t+1},z^{t+1},x^t) - \D_\gamma(y^{t+1},z^t,x^t) \le -\frac1\gamma \|x^{t+1} - x^t\|^2 + \frac1\gamma \left((-1+2\gamma l)\|y^{t+1} - y^t\|^2 +  \|x^t - x^{t-1}\|^2\right).
  \end{equation}

  Finally, using the second relation in \eqref{Drelation}, we obtain that
  \begin{equation}\label{D3ineq}
    \begin{split}
      \D_\gamma(y^{t+1},z^t,x^t) - \D_\gamma(y^t,z^t,x^t)
      & = f(y^{t+1}) + \frac1{2\gamma} \|x^t - y^{t+1}\|^2 - f(y^t) - \frac1{2\gamma} \|x^t - y^t\|^2\\
      & \le -\frac{1}{2}\left(\frac1\gamma - l\right)\|y^{t+1} - y^t\|^2,
    \end{split}
  \end{equation}
  where the  inequality follows from the fact that $f+\frac1{2\gamma} \|x^t - \cdot\|^2$ is a strongly convex function with modulus $\frac{1}{\gamma}-l$ and the definition of $y^{t+1}$ as a minimizer.  
  Summing \eqref{D1ineq}, \eqref{D2ineq2} and \eqref{D3ineq}, we see further that for any $t\ge 1$,
  \begin{equation}\label{usefulrelation1}
    \D_\gamma(y^{t+1},z^{t+1},x^{t+1}) - \D_\gamma(y^t,z^t,x^t) \le \frac{-3+5\gamma l}{2\gamma}\|y^{t+1} - y^{t}\|^2 +\frac1\gamma\|x^t - x^{t-1}\|^2.
  \end{equation}
  Since we also have from the first relation in \eqref{fyupdate} and the Lipschitz continuity of $\nabla f$ that for $t\ge 1$
  \begin{equation}\label{ineqbd}
  \|x^t - x^{t-1}\| \le (1 + \gamma L)\|y^{t+1} - y^t\|,
  \end{equation}
  we conclude further that for any $t\ge 1$
  \begin{equation}\label{useful2}
    \D_\gamma(y^{t+1},z^{t+1},x^{t+1}) - \D_\gamma(y^t,z^t,x^t) \le \frac{1}{\gamma}\left((1 + \gamma L)^2 + \frac{5\gamma l}2 - \frac32\right)\|y^{t+1} - y^t\|^2.
  \end{equation}
  Since $(1 + \gamma L)^2 + \frac{5\gamma  l}2  - \frac{3}{2} < 0$ by our choice of $\gamma$, we see that $\{\D_\gamma(y^t,z^t,x^t)\}_{t\ge 1}$ is nonincreasing.

  Summing \eqref{useful2} from $t=1$ to $N-1\ge 1$, we obtain that
  \begin{equation}\label{keybd}
    \D_\gamma(y^N,z^N,x^N) - \D_\gamma(y^1,z^1,x^1) \le \frac1\gamma\left( (1 + \gamma L)^2+\frac{5\gamma l}2 - \frac{3}{2} \right)\sum_{t=1}^{N-1}\|y^{t+1} - y^t\|^2.
  \end{equation}
  Hence, if a cluster point $(y^*,z^*,x^*)$ exists with a convergent subsequence
  $\lim_{j\rightarrow \infty}(y^{t_j},z^{t_j},x^{t_j}) = (y^*,z^*,x^*)$, then using the lower semi-continuity of
  $\D_\gamma$ and taking limit as $j\rightarrow \infty$ with $N = t_j$ in \eqref{keybd}, we have
  \begin{equation*}
    -\infty < \D_\gamma(y^*,z^*,x^*) - \D_\gamma(y^1,z^1,x^1) \le \frac1\gamma\left( (1 + \gamma L)^2+\frac{5\gamma l}2 - \frac{3}{2} \right)\sum_{t=1}^{\infty}\|y^{t+1} - y^t\|^2,
  \end{equation*}
  where the first inequality follows from the fact that $\D_\gamma$ is proper.
  From this we conclude immediately that $\lim_{t\rightarrow \infty}\|y^{t+1} - y^t\|=0$. Combining this with \eqref{ineqbd}, we conclude that \eqref{cond1} holds.
  Furthermore, combining these with the third relation in \eqref{scheme}, we obtain further that $\lim_{t\rightarrow \infty}\|z^{t+1} - z^t\|=0$.
  Thus, if $(y^*,z^*,x^*)$ is a cluster point of $\{(y^t,z^t,x^t)\}$ with a convergent subsequence $\{(y^{t_j},z^{t_j},x^{t_j})\}$ so that
  $\lim_{j\rightarrow \infty}(y^{t_j},z^{t_j},x^{t_j}) = (y^*,z^*,x^*)$, then
  \begin{equation}\label{limit}
  \lim_{j\rightarrow \infty}(y^{t_j},z^{t_j},x^{t_j}) =
  \lim_{j\rightarrow \infty}(y^{t_j-1},z^{t_j-1},x^{t_j-1}) = (y^*,z^*,x^*).
  \end{equation}
  From the definition of $z^t$ as a minimizer, we have
  \begin{equation}\label{glimit}
  g(z^t) + \frac1{2\gamma}\|2y^t - z^t - x^{t-1}\|^2 \le g(z^*) + \frac1{2\gamma}\|2y^t - z^* - x^{t-1}\|^2.
  \end{equation}
  Taking limit along the convergent subsequence
  and using \eqref{limit} yields
  \begin{equation}\label{glimit2}
  \limsup_{j\rightarrow \infty}g(z^{t_j})\le g(z^*).
  \end{equation}
  On the other hand, by the lower semicontinuity of $g$, we have $\liminf_{j\rightarrow \infty}g(z^{t_j})\ge g(z^*)$.
  Consequently, \eqref{cond2} holds.
  Now passing to the limit in \eqref{eq:stationary} along the convergent subsequence $\{(y^{t_j},z^{t_j},x^{t_j})\}$, and using \eqref{cond1}, \eqref{cond2} and \eqref{outersemi}, we see that
  the conclusion of the theorem follows.
\end{proof}

Under the additional assumption that the functions $f$ and $g$ are semi-algebraic functions, we now show that, in the next theorem that, if the whole sequence generated has a cluster point, then it is actually convergent.
The argument is largely inspired from the proof of \cite[Lemma~2.6]{AtBoSv13} with suitable modifications.

\begin{theorem}\label{thm1.5} {\bf (Global convergence of the whole sequence)}
  Suppose that the step-size parameter $\gamma > 0$ is chosen as in \eqref{gammacond}
  and the sequence $\{(y^t,z^t,x^t)\}$ generated has a cluster point $(y^*,z^*,x^*)$. Suppose in addition that $f$ and $g$ are semi-algebraic functions.
  Then the whole sequence $\{(y^t,z^t,x^t)\}$ is convergent.
\end{theorem}
\begin{proof}

  We first consider the subdifferential of $\D_\gamma$ at $(y^{t+1},z^{t+1},x^{t+1})$. Notice that
  for any $t\ge 0$, we have
  \begin{equation*}
    \begin{split}
      \nabla_x \D_\gamma(y^{t+1},z^{t+1},x^{t+1}) &= \frac1\gamma (z^{t+1} - y^{t+1}) = \frac1\gamma (x^{t+1} - x^t),\\
      \nabla_y \D_\gamma(y^{t+1},z^{t+1},x^{t+1}) &= \nabla f(y^{t+1}) + \frac1{\gamma}(y^{t+1} - x^{t+1}) = \frac1\gamma(x^t - x^{t+1}),
    \end{split}
  \end{equation*}
  where for the first gradient we made use of \eqref{Ddef} and the definition of $x^{t+1}$, while for the second gradient we made use of the second relation in \eqref{Drelation} and the first relation in \eqref{fyupdate}. Moreover, for the subdifferential with respect to $z$, we have from the second relation in \eqref{Drelation} that
  \begin{equation*}
   \begin{split}
    & \partial_z \D_\gamma(y^{t+1},z^{t+1},x^{t+1}) = \partial g(z^{t+1}) - \frac1\gamma (z^{t+1} - x^{t+1})\\
    & = \partial g(z^{t+1}) + \frac1\gamma(z^{t+1} - y^{t+1}) - \frac1\gamma(y^{t+1}-x^t) - \frac1\gamma(z^{t+1} - y^{t+1}) + \frac1\gamma(y^{t+1}-x^t) - \frac1\gamma (z^{t+1} - x^{t+1})\\
    & \ni -\frac2\gamma(z^{t+1} - y^{t+1}) + \frac1\gamma(x^{t+1}-x^t) = -\frac1\gamma(x^{t+1}-x^t),
   \end{split}
  \end{equation*}
  where the inclusion follows from the second relation in \eqref{fyupdate}, and the last equality follows from the definition of $x^{t+1}$.
  The above relations together with \eqref{ineqbd} imply the existence of $\tau>0$
  so that whenever $t\ge 1$, we have
  \begin{equation}\label{gradientbd}
    {\rm dist}(0,\partial \D_\gamma(y^t,z^t,x^t)) \le \tau \|y^{t+1} - y^t\|.
  \end{equation}
  On the other hand, notice from \eqref{useful2} that there exists $K > 0$ so that
  \begin{equation}\label{Dgammadecrease}
    \D_\gamma(y^t,z^t,x^t) - \D_\gamma(y^{t+1},z^{t+1},x^{t+1}) \ge K\|y^{t+1} - y^t\|^2.
  \end{equation}
  In particular, $\{\D_\gamma(y^t,z^t,x^t)\}$ is non-increasing. Let $\{(y^{t_i},z^{t_i},x^{t_i})\}$ be a convergent subsequence that converges to
  $(y^*,z^*,x^*)$. Then, from the lower semicontinuity of $\D_{\gamma}$, we see that the sequence $\{\D_\gamma(y^{t_i},z^{t_i},x^{t_i})\}$  is bounded below. This together with the non-increasing property of $\{\D_\gamma(y^t,z^t,x^t)\}$ shows that $\{\D_\gamma(y^t,z^t,x^t)\}$ is also bounded below, and so,
  %
  $\lim_{t\rightarrow \infty}\D_\gamma(y^t,z^t,x^t) = l^*$ exists.

  We next claim that $l^*= \D_\gamma(y^*,z^*,x^*)$. Let $\{(y^{t_j},z^{t_j},x^{t_j})\}$ be any subsequence that converges to $(y^*,z^*,x^*)$. Then from lower semicontinuity, we readily have
  \begin{equation*}
    \liminf_{j\rightarrow \infty}\D_\gamma(y^{t_j},z^{t_j},x^{t_j}) \ge \D_\gamma(y^*,z^*,x^*).
  \end{equation*}
  On the other hand, proceeding as in \eqref{limit}, \eqref{glimit} and \eqref{glimit2}, we can conclude further that
  \begin{equation*}
    \limsup_{j\rightarrow \infty}\D_\gamma(y^{t_j},z^{t_j},x^{t_j}) \le \D_\gamma(y^*,z^*,x^*).
  \end{equation*}
  These together with the existence of $\lim_{t\rightarrow \infty}\D_\gamma(y^t,z^t,x^t)$ shows that $l^*= \D_\gamma(y^*,z^*,x^*)$, as claimed.
  Note that if $\D_\gamma(y^t,z^t,x^t) = l^*$ for some $t\ge 1$, then $\D_\gamma(y^t,z^t,x^t) = \D_\gamma(y^{t+k},z^{t+k},x^{t+k})$ for all $k\ge 0$ since the sequence is non-increasing. Then \eqref{Dgammadecrease} gives $y^t = y^{t+k}$ for all $k\ge 0$. From \eqref{ineqbd}, we see that $x^t = x^{t+k}$ for $k\ge 0$. These together with the third relation in \eqref{scheme} show that we also have $z^{t+1} = z^{t+k}$ for $k\ge 1$. Thus, the sequence remains constant from the $(t+1)$st iteration onward. Since this theorem holds trivially when this happens, from now on, we assume $\D_\gamma(y^t,z^t,x^t) > l^*$ for all $t\ge 1$.

  Next, from \cite[Section~4.3]{AtBoReSo10} and our assumption on semi-algebraicity, the function $(y,z,x)\mapsto \D_\gamma(y,z,x)$ is a KL function. From the property of KL functions, there exist $\nu > 0$, a neighborhood ${\cal V}$ of $(y^*,z^*,x^*)$ and a continuous concave function $\psi:[0,\nu)\rightarrow \R_+$ as described in Definition~\ref{def:KL} so that
  for all $(y,z,x)\in {\cal V}$ satisfying $l^*< \D_\gamma(y,z,x) < l^* + \nu$, we have
    \begin{equation}\label{eq:KL_ineq2}
      \psi'(\D_\gamma(y,z,x) - l^*)\,{\rm dist}(0,\partial \D_\gamma(y,z,x))\ge 1.
    \end{equation}

  Pick $\rho > 0$ so that
  \[
  {\bf B}_\rho:=\left\{(y,z,x):\; \|y - y^*\|<\rho,\ \|x - x^*\| < (2 + \gamma L)\rho,\ \|z-z^*\| < 2\rho\right\}\subseteq {\cal V}
  \]
  and set $B_\rho:= \{y:\; \|y - y^*\| < \rho\}$.
  Observe from the first relation in \eqref{fyupdate} that
  \begin{equation*}
    \|x^t - x^*\| \le \|x^t - x^{t-1}\| + \|x^{t-1} - x^*\| \le \|x^t - x^{t-1}\| + (1 + \gamma L)\|y^t - y^*\|.
  \end{equation*}
  Since \eqref{cond1} holds by Theorem~\ref{thm1}, there exists $N_0\ge 1$ so that $\|x^t - x^{t-1}\| < \rho$ whenever $t\ge N_0$. Thus, it follows that $\|x^t - x^*\|< (2 + \gamma L)\rho$ whenever $y^t \in B_\rho$
  and $t\ge N_0$.
  Next, using the third relation in \eqref{scheme}, we see also that for all $t\ge N_0$,
  \begin{equation*}
    \|z^t - z^*\| \le \|y^t - y^*\| + \|x^t - x^{t-1}\| < 2\rho
  \end{equation*}
  whenever $y^t \in B_\rho$. Consequently, if $y^t\in B_\rho$ and $t\ge N_0$, then $(y^t,z^t,x^t)\in {\bf B}_\rho\subseteq {\cal V}$.
  Furthermore, using the facts that $(y^*,z^*,x^*)$ is a cluster point, that $\lim_{t\rightarrow \infty}\D_\gamma(y^t,z^t,x^t) = l^*$, and that $\D_\gamma(y^t,z^t,x^t)>l^*$ for all $t\ge 1$,
  it is not hard to see that there exists $(y^N,z^N,x^N)$ with $N \ge N_0$ such that
  \begin{enumerate}[{\rm (i)}]
    \item $y^N\in B_\rho$ and $l^*< \D_\gamma(y^N,z^N,x^N) < l^* + \nu$;
    \item $\|y^N - y^*\| + \frac{\tau}{K}\psi(\D_\gamma(y^{N},z^{N},x^{N}) - l^*) < \rho$.
  \end{enumerate}

  Before proceeding further, we show that whenever $y^t\in B_\rho$ and $l^* < \D_\gamma(y^t,z^t,x^t) < l^* + \nu$ for some fixed $t\ge N_0$, we have
  \begin{equation}\label{xbd}
    \|y^{t+1} - y^t\| \le \frac \tau K [\psi(\D_\gamma(y^t,z^t,x^t) - l^*) - \psi(\D_\gamma(y^{t+1},z^{t+1},x^{t+1}) - l^*)].
  \end{equation}
  Since $\{\D_\gamma(y^t,z^t,x^t)\}$ is non-increasing and $\psi$ is increasing, \eqref{xbd} clearly holds if $y^{t+1} = y^t$. Hence, suppose without loss of generality that $y^{t+1} \neq y^t$. Since $y^t\in B_\rho$ and $t\ge N_0$, we have $(y^t,z^t,x^t)\in {\bf B}_\rho\subseteq {\cal V}$. Hence, \eqref{eq:KL_ineq2} holds for $(y^t,z^t,x^t)$. Making use of
  the concavity of $\psi$, \eqref{gradientbd}, \eqref{Dgammadecrease} and \eqref{eq:KL_ineq2}, we see that for all such $t$
  \begin{equation*}
    \begin{split}
      & \tau \|y^{t+1} - y^t\|\cdot[\psi(\D_\gamma(y^t,z^t,x^t) - l^*) - \psi(\D_\gamma(y^{t+1},z^{t+1},x^{t+1}) - l^*)]\\
      &\ge  {\rm dist}(0,\partial \D_\gamma(y^t,z^t,x^t))\cdot[\psi(\D_\gamma(y^t,z^t,x^t) - l^*) - \psi(\D_\gamma(y^{t+1},z^{t+1},x^{t+1}) - l^*)]\\
      &\ge  {\rm dist}(0,\partial \D_\gamma(y^t,z^t,x^t))\cdot\psi'(\D_\gamma(y^t,z^t,x^t) - l^*)\cdot[\D_\gamma(y^t,z^t,x^t) - \D_\gamma(y^{t+1},z^{t+1},x^{t+1})]\\
      &\ge  K \|y^{t+1} - y^t\|^2,
    \end{split}
  \end{equation*}
  from which \eqref{xbd} follows immediately.

  We next show that $y^t\in B_\rho$ whenever $t\ge N$ by induction.
  The claim is true for $t=N$ by construction.
  Now, suppose the claim is true for $t=N,\ldots,N+k-1$ for some $k\ge 1$; i.e., $y^N,\ldots,y^{N+k-1}\in B_\rho$. Notice that as $\{\D_\gamma(y^t,z^t,x^t)\}$ is a non-increasing sequence, our choice of $N$ implies that $l^*<\D_\gamma(y^t,z^t,x^t)<l^* + \nu$ for all $t \ge N$. In particular, \eqref{xbd} can be applied for $t = N,\ldots,N+k-1$. Thus, for $t = N+k$, we have from this observation that
  \begin{equation*}
    \begin{split}
      &\|y^{N+k} - y^*\| \le \|y^N - y^*\| + \sum_{j=1}^{k}\|y^{N+j}-y^{N+j-1}\|\\
      & \le \|y^N - y^*\| + \frac \tau K\sum_{j=1}^{k}[\psi(\D_\gamma(y^{N+j-1},z^{N+j-1},x^{N+j-1}) - l^*) - \psi(\D_\gamma(y^{N+j},z^{N+j},x^{N+j}) - l^*)]\\
      & \le \|y^N - y^*\| + \frac \tau K\psi(\D_\gamma(y^{N},z^{N},x^{N}) - l^*) < \rho,
    \end{split}
  \end{equation*}
  where the first inequality in the last line follows from the nonnegativity of $\psi$. Thus, we have shown that $y^t\in B_\rho$ for $t\ge N$ by induction.

  Since $y^t\in B_\rho$ and $l^*<\D_\gamma(y^t,z^t,x^t)<l^* + \nu$ for $t\ge N$, we can sum \eqref{xbd} from $t=N$ to $M\rightarrow \infty$, showing that $\{\|y^{t+1}-y^t\|\}$ is summable. Convergence of $\{y^t\}$ follows immediately from this. Convergence of $\{x^t\}$ follows from this and the first relation in \eqref{fyupdate}. Finally, the convergence of $\{z^t\}$ follows from the third relation in \eqref{scheme}. This completes the proof.
\end{proof}

\begin{remark}{\bf (Comments on the proof)} Below, we make some comments about the proof of Theorem \ref{thm1.5}.
\begin{enumerate}[{\rm (i)}]
\item Our proof indeed shows that, if the assumptions in Theorem \ref{thm1.5} hold, then the sequence $\{(y^t,z^t,x^t)\}$ generated by the DR splitting method has a finite length, i.e., $$\sum_{t=1}^{\infty}(\|y^{t+1}-y^t\|+\|z^{t+1}-z^t\|+\|x^{t+1}-x^t\|)<+\infty.$$
Precisely, the summability of $\|y^{t+1}-y^{t}\|$ and $\|x^{t+1}-x^t\|$ can be seen from \eqref{xbd} and \eqref{ineqbd}. Moreover, notice from the third relation in \eqref{scheme} that
\[
\begin{split}
\|z^{t+1}-z^t\| & = \|(y^{t+1} + x^{t+1} - x^t) - (y^t + x^t - x^{t-1})\|\\
&\le  \|y^{t+1} - y^t\|+\|x^{t+1} - x^t\|+\|x^t - x^{t-1}\|.
\end{split}
\]
Therefore, the summability of $\|z^{t+1}-z^{t}\|$ follows from the summability of  $\|y^{t+1}-y^{t}\|$ and $\|x^{t+1}-x^t\|$.
 \item  The proof of Theorem~\ref{thm1.5} stays valid as long as the DR merit function $\D_\gamma$ is a KL-function.
We only state the case where $f$ and $g$ are semi-algebraic as this simple sufficient condition can be readily checked.
\end{enumerate}

\end{remark}

Recall from Proposition~\ref{Prop1} that a semi-algebraic function $h$ satisfies the KL inequality with $\psi(s)=c\, s^{1-\theta}$ for some $\theta \in [0,1)$ and $c>0$. We now derive eventual convergence rates of the proposed nonconvex DR splitting method by examining the range of the exponent.

\begin{theorem}{\bf (Eventual convergence rate)}\label{thm3}
Suppose that the step-size parameter $\gamma > 0$ is chosen as in \eqref{gammacond}
  and the sequence $\{(y^t,z^t,x^t)\}$ generated has a cluster point $(y^*,z^*,x^*)$. Suppose in addition that $f$ and $g$ are semi-algebraic functions
  so that the $\psi$ in the KL inequality \eqref{eq:KL_ineq2} takes the form $\psi(s)=c\, s^{1-\theta}$ for some $\theta \in [0,1)$ and $c>0$. Then, we have
\begin{itemize}
\item[{\rm (i)}]  If $\theta=0$, then there exists $t_0 \ge 1$ such that for all $t \ge t_0$, $0 \in \nabla f(z^{t})+\partial g(z^{t})$;
\item[{\rm (ii)}] If $\theta \in (0,\frac{1}{2}]$, then there exist $\eta \in (0,1)$ and $\kappa>0$ so that ${\rm dist}\big(0,\nabla f(z^t)+\partial g(z^t)\big) \le \kappa \, \eta^t$ for all large $t$;
\item[{\rm (iii)}]  If $\theta \in (\frac{1}{2},1)$, then there exists $\kappa > 0$ such that ${\rm dist}\big(0,\nabla f(z^t)+\partial g(z^t)\big) \le \kappa \, t^{-\frac1{4\theta-2}}$ for all large $t$.
\end{itemize}
\end{theorem}
\begin{proof}
Let $l_t=\D_\gamma(y^t,z^t,x^t) - \D_\gamma(y^*,z^*,x^*)$. Then, we have from the proof of Theorem~\ref{thm1.5} that $l_t \ge 0$ for all $t\ge 1$ and $l_t \rightarrow 0$ as $t \rightarrow \infty$. Furthermore, from \eqref{Dgammadecrease}, we have
 \begin{equation}\label{eq:use1}
l_{t}-l_{t+1} \ge K\|y^{t+1} - y^{t}\|^2.
 \end{equation}
As $l_{t+1} \ge 0$, it follows that  $K\|y^{t+1} - y^{t}\|^2 \le l_{t}-l_{t+1} \le l_{t}$ for all $t \ge 1$. This together with \eqref{ineqbd} implies that
\[
\|y^{t}-z^{t}\|=\|x^{t} - x^{t-1}\| \le (1+\gamma L) \|y^{t+1}-y^{t}\| \le \frac{(1+\gamma L)}{\sqrt{K}} \sqrt{l_t},
\]
where the first equality follows from the last relation in \eqref{scheme}. Notice from \eqref{eq:stationary} and the Lipschitz continuity of $\nabla f$ that
\[
{\rm dist}(0,\nabla f(z^t) + \partial g(z^t)) \le \left(L + \frac{1}{\gamma}\right)\|y^t-z^t\|.
\]
Consequently, for all $t \ge 1$,
\begin{equation}\label{eq:distance_bd}
{\rm dist}(0,\nabla f(z^t) + \partial g(z^t))  \le \frac{(1+\gamma L)^2}{\gamma \sqrt{K}} \sqrt{l_t}.
\end{equation}
Moreover, from the convergence of $\{(y^t,z^t,x^t)\}$ to $(y^*,z^*,x^*)$ guaranteed by Theorem~\ref{thm1.5}, the relation \eqref{eq:KL_ineq2}, and the discussion that precedes it, we see that either
\begin{enumerate}[{\rm Case} (i):]
  \item there exists a $t_0 \ge 1$ such that $l_t=0$ for some and hence all $t \ge t_0$; or
  \item for all large $t$, we have $l_t > 0$ and
\begin{equation}\label{eq:KL_ineq3}
c\, (1-\theta) l_t^{-\theta} {\rm dist}\left(0,\partial \D_\gamma(y^t,z^t,x^t)\right) \ge 1.
\end{equation}
\end{enumerate}
For Case (i), \eqref{eq:distance_bd} implies that the conclusion follows trivially. Therefore, we consider Case (ii).
From \eqref{gradientbd}, we obtain that
${\rm dist}(0,\partial \D_\gamma(y^t,z^t,x^t)) \le \tau\|y^{t+1} - y^{t}\|$.
It then follows that
\begin{equation}\label{eq:use2}
\|y^{t+1} - y^{t}\| \ge \frac{1}{c\, (1-\theta)\tau} \, l_t^{\theta}.
\end{equation}
Therefore, combining \eqref{eq:use1} and \eqref{eq:use2}, we see that
\[
l_{t}-l_{t+1} \ge M l_t^{2 \theta} \mbox{ for all large } t,
\]
where $M=K (\frac{1}{c\, (1-\theta)\tau})^2$. We now divide the discussion into three cases:

Case 1: $\theta =0$. In this case, we have $l_{t}-l_{t+1} \ge M>0$, which contradicts $l_t\rightarrow 0$. Thus, this case cannot happen.

Case 2: $\theta \in (0,\frac{1}{2}]$. In this case, as $l_t \rightarrow 0$, there exists $t_1 \ge 1$ such that $l_t^{2\theta} \ge l_t$ for all $t \ge t_1$. Then, for all large $t$
\begin{equation}\label{Q-linear}
l_{t+1} \le l_t- M l_t^{2 \theta} \le (1-M) l_t.
\end{equation}
From this and the positivity of $l_t$, we see immediately that $1 - M > 0$ and
that there exists $\mu>0$ such that  $l_t \le \mu (1-M)^{t}$ for all large $t$. This together with \eqref{eq:distance_bd}
 implies that the conclusion of Case 2 follows with $\kappa=\frac{\sqrt{\mu}(1+\gamma L)^2}{\gamma \sqrt{K}}$ and $\eta=\sqrt{1-M} \in (0,1)$.

Case 3: $\theta \in (\frac{1}{2},1)$. Define the non-increasing function $h:(0,+\infty) \rightarrow\R$ by   $h(s):=s^{-2\theta}$. As there exists $i_0 \ge 1$ such that $ M \, h(l_i)^{-1}= M l_i^{2\theta} \le l_i-l_{i+1}$ for all $i \ge i_0$,
then we get, for all $i \ge i_0$,
\[
M \le
(l_{i}-l_{i+1})h(l_i)\le
\int_{l_{i+1}}^{l_i}h(s)ds=
\frac{l_{i}^{1-2\theta}-l_{i+1}^{1-2\theta}}{1-2\theta}=\frac{l_{i+1}^{1-2\theta}-l_{i}^{1-2\theta}}{2\theta-1}.
\]
Noting that $2\theta-1 >0$, this implies that for all $i \ge i_0$
\begin{equation*}
l_{i+1}^{1-2\theta}-l_{i}^{1-2\theta}\ge M (2\theta-1).
\end{equation*}
Summing for all $i=i_0$ to $i=t-1$ we have for all large $t$
\[
l_{t}^{1-2\theta}-l_{i_0}^{1-2\theta} \ge M(2\theta-1) (t-i_0).
\]
This gives us that for all large $t$
\[
l_t \le \frac{1}{\sqrt[2\theta-1]{l_{i_0}^{1-2\theta}+M(2\theta-1) (t-i_0)}} \ .
\]
So, combining this with \eqref{eq:distance_bd}, we see that the conclusion of Case 3 follows.
\end{proof}

\begin{remark}{\bf (Comments on Theorem~\ref{thm3})}\label{rem:thm3}
  \begin{enumerate}[{\rm (i)}]
    \item A closer look at the proof of Theorem~\ref{thm3} reveals that one only needs \eqref{eq:KL_ineq2} to hold with $\psi(s)=c\, s^{1-\theta}$ for some $\theta \in [0,1)$ and $c>0$ along the sequence $\{(y^t,z^t,x^t)\}$ generated by \eqref{scheme} for all sufficiently large $t$.
    \item For Theorem~\ref{thm3} to be informative for a particular instance, one has to give an explicit estimate of $\theta$. As an example, we will show in Proposition~\ref{prop:1} below that under a constraint qualification, we can have $\theta=\frac12$ for the feasibility problem under consideration.
    Moreover, there is some recent work devoted to providing an explicit estimate of $\theta$ in the KL inequality when the semi-algebraic
    function has a specific structure. For example, if the DR merit function can be expressed as a maximum of finitely many polynomials $h_i$, $i=1,\ldots,q$, (this happens when $f$ and $g$ can be expressed as maximums of finitely many polynomials),
    then \cite[Theorem 3.3]{Li_Bor_Pham} provides an explicit estimate of the exponent $\theta$ in terms of the degrees of the polynomials $h_i$ and the dimension
    of the underlying space.
  \end{enumerate}
\end{remark}

All our preceding convergence results rely on the existence of a cluster point. Before ending this section, we give some simple sufficient conditions that will guarantee the sequence generated from the DR splitting method is bounded.
As we will see in the next section, these simple sufficient conditions can be easily satisfied for nonconvex feasibility problem under mild assumptions.

\begin{theorem}{\bf (Boundedness of the sequence generated from the DR splitting method)}\label{thm3.5}
  Suppose that $\gamma$ is chosen to satisfy \eqref{gammacond}. Suppose in addition that $f$ and $g$ are both bounded from below, and
  that at least one of them is coercive.
  Then the sequence $\{(y^t,z^t,x^t)\}$ generated from \eqref{scheme} is bounded.
\end{theorem}
\begin{proof}
  Since $f$ is bounded from below, say, by $\zeta_* > -\infty$, we have for any $x$ that
  \begin{equation}\label{gradcond}
    \begin{split}
      \zeta_*  \le f\left(x - \frac1L \nabla f(x)\right) &\le f(x) + \left\langle\nabla f(x), \left(x - \frac1L \nabla f(x)\right) - x\right\rangle + \frac{L}2\left\|\left(x - \frac1L \nabla f(x)\right) - x\right\|^2\\
      & = f(x) - \frac{1}{2L}\|\nabla f(x)\|^2,
    \end{split}
  \end{equation}
  where the second inequality follows from the Taylor series expansion and the Lipschitz continuity of the gradient of $f$.
  Next, we have from the assumption on $\gamma$ and Theorem~\ref{thm1} that for all $t\ge 1$,
  \begin{equation}\label{bd2}
    \D_\gamma(y^t,z^t,x^t) \le \D_\gamma(y^1,z^1,x^1).
  \end{equation}
  In addition, using the second relation in \eqref{Drelation}, we have for $t\ge 1$ that
  \begin{equation}\label{ineq22}
  \begin{split}
   \D_\gamma(y^t,z^t,x^t)
    & = f(y^t) + g(z^t) - \frac1{2\gamma}\|x^t - z^t\|^2 + \frac1{2\gamma}\|x^t - y^t\|^2\\
    & = f(y^t) + g(z^t) - \frac1{2\gamma}\|x^{t-1} - y^t\|^2 + \frac1{2\gamma}\|x^t - y^t\|^2,
  \end{split}
  \end{equation}
  where the last equality follows from the definition of $x^{t+1}$, i.e., the third relation in \eqref{scheme}.
  Moreover, from the first relation in \eqref{fyupdate}, we have for $t\ge 1$ that
  \begin{equation}\label{relation}
  0 = \nabla f(y^t) + \frac1\gamma(y^t - x^{t-1}),
  \end{equation}
  which implies that $\|x^{t-1} - y^t\|^2 = \gamma^2\|\nabla f(y^t)\|^2$. Furthermore, notice that because of \eqref{gammaupperbound}, we can choose $\mu\in (0,1)$ so that $\frac{1-\mu}L > \gamma$. Combining these with \eqref{ineq22} and \eqref{bd2}, we obtain further that
  \begin{equation}\label{keyineq22}
  \begin{split}
   \D_\gamma(y^1,z^1,x^1) &\ge \D_\gamma(y^t,z^t,x^t) = f(y^t) + g(z^t) - \frac1{2\gamma}\|x^{t-1} - y^t\|^2 + \frac1{2\gamma}\|x^t - y^t\|^2\\
  & = f(y^t) - \frac\gamma{2}\|\nabla f(y^t)\|^2 + g(z^t) + \frac1{2\gamma}\|x^t - y^t\|^2\\
  & = \mu f(y^t) + (1-\mu)f(y^t) - \frac\gamma{2}\|\nabla f(y^t)\|^2 + g(z^t) + \frac1{2\gamma}\|x^t - y^t\|^2\\
  & \ge \mu f(y^t) + (1 - \mu)\zeta_* + \frac12\left(\frac{1- \mu}{L} - \gamma \right)\|\nabla f(y^t)\|^2 + g(z^t) + \frac1{2\gamma}\|x^t - y^t\|^2,
  \end{split}
  \end{equation}
  where the last inequality follows from \eqref{gradcond}.

  Now, suppose first that $g$ is coercive. Then it follows readily from \eqref{keyineq22} that $\{z^t\}$, $\{\nabla f(y^t)\}$ and $\{x^t - y^t\}$ are bounded. From \eqref{relation} we see immediately that $\{y^t - x^{t-1}\}$ is also bounded. This together with the boundedness of $\{x^t - y^t\}$ shows that $\{x^t - x^{t-1}\}$ is also bounded. From the third relation in \eqref{scheme}, this means that $\{z^t - y^t\}$ is bounded. Since we know already that $\{z^t\}$ is bounded, it follows that $\{y^t\}$ is also bounded. The boundedness of $\{x^t\}$ now follows from this and the boundedness of $\{x^t - y^t\}$.

  Finally, suppose that $f$ is coercive. Then we see immediately from \eqref{keyineq22} that $\{y^t\}$ and $\{x^t - y^t\}$ are bounded. Consequently, the sequence $\{x^t\}$ is also bounded. The boundedness of $\{z^t\}$ then follows from the third relation in \eqref{scheme}. This completes the proof.
\end{proof}


\section{Douglas-Rachford splitting for nonconvex feasibility problems}\label{sec4}

In this section, we discuss how the nonconvex DR splitting method in Section~\ref{sec3} can be applied to solving a feasibility problem.

Let $C$ and $D$ be two nonempty closed sets, with $C$ being convex. We also assume that a projection onto each of them is easy to compute. The feasibility problem is to find a point in $C\cap D$, if any.
It is clear that $C\cap D\neq \emptyset$ if and only if the following optimization problem has a zero optimal value:
\begin{equation}\label{prob:feas}
  \begin{array}{rl}
    \min\limits_u & \frac12 d_C^2(u)\\
    {\rm s.t.} & u\in D.
  \end{array}
\end{equation}
Since $C$ is closed and convex, it is well known that the function $u\mapsto \frac12 d_C^2(u)$ is smooth with a Lipschitz continuous gradient whose Lipschitz continuity modulus is $1$; see, for example, \cite[Corollary~12.30]{BauCom11}. \footnote{We note that, a more general
and informative statement that applies to the square distance function of a possibly nonconvex but prox-regular set can be found in \cite[Theorem 1.3]{PRT00}.} Moreover, for each $\gamma > 0$, one can observe that
\begin{equation*}
    \inf_y\left\{\frac12 d_C^2(y) + \frac1{2\gamma}\|y - x\|^2\right\}  = \inf_{c\in C}\inf_y\left\{\frac12\|y - c\|^2 + \frac1{2\gamma}\|y - x\|^2\right\}.
\end{equation*}
The first-order optimality condition of the inner optimization problem on the right gives $y = \frac{x + \gamma c}{1 + \gamma}$. Using this expression, one can further simplify the expression on the right to obtain the following:
\begin{equation*}
  \inf_y\left\{\frac12 d_C^2(y) + \frac1{2\gamma}\|y - x\|^2\right\} = \inf_{c\in C}\frac{1}{2(1 + \gamma)}\|x - c\|^2,
\end{equation*}
with the infimum on the right attained at $c = P_C(x)$. Consequently, we have shown that
\[
\frac{1}{1+\gamma}(x + \gamma P_C(x)) = \argmin\limits_{y}\left\{\frac12 d_C^2(y) + \frac1{2\gamma}\|y - x\|^2\right\}.
\]
Hence, applying the DR splitting method in Section~\ref{sec3} to solving \eqref{prob:feas} gives the following algorithm:\\ \

\fbox{\parbox{5.7 in}{
\begin{description}
\item {\bf Douglas-Rachford splitting method for feasibility problem}

\item[Step 0.] Input an initial point $x^0$ and a step-size parameter $\gamma > 0$.

\item[Step 1.] Set
\begin{equation}\label{scheme2}
\left\{
\begin{split}
&y^{t+1} = \frac{1}{1+\gamma}(x^t + \gamma P_C(x^t)), \\
&z^{t+1}\in\Argmin_{z\in D} \left\{\|2y^{t+1} - x^t - z\|^2\right\},\\
&x^{t+1}=x^t+(z^{t+1}- y^{t+1}).
\end{split}
\right.
\end{equation}

\item[Step 2.] If a termination criterion is not met, go to Step 1.
\end{description}
}}
\vspace{.1 in}

Notice that the above algorithm involves a computation of $P_C(x^t)$ and a projection of $2y^{t+1}-x^t$ onto $D$, which are both easy to compute by assumption.
 Moreover, observe that as $\gamma\rightarrow \infty$, \eqref{scheme2} reduces to the classical DR splitting method considered in the literature for finding
 a point in $C\cap D$, i.e., the DR splitting method in \eqref{scheme} applied to minimizing the sum of the indicator functions of $C$ and $D$.
 Comparing with this, the version in \eqref{scheme2} can be viewed as a {\em damped} DR splitting method for finding feasible points.
 It is also worth noting that \eqref{scheme2} was
studied in \cite{Luke08} where the author showed that this algorithm is equivalent to the relaxed averaged alternating
reflections algorithm with a suitable choice of the parameter.

The global convergence of both the classical DR splitting method and \eqref{scheme2} are known in the convex scenario, i.e, when $D$ is also convex.
However, in our case, $C$ is closed and convex while $D$
is possibly nonconvex. Thus, the known results do not apply directly.
Nonetheless, we have the following convergence result of \eqref{scheme2} using Theorem~\ref{thm1}. In addition, we can show in this particular case that the sequence $\{(y^t,z^t,x^t)\}$ generated from \eqref{scheme2} is bounded, assuming $C$ or $D$ is compact.

\begin{theorem}\label{thm2} {\bf (Convergence of DR splitting method for nonconvex feasibility problem involving two sets)}
  Suppose that $C$ is a nonempty closed convex set and $D$ is a nonempty closed set, and that either $C$ or $D$ is compact. Suppose in addition that $0<\gamma < \sqrt{\frac32} - 1$. Then the sequence $\{(y^t,z^t,x^t)\}$ generated from \eqref{scheme2} is bounded,
  and any cluster point $(y^*,z^*,x^*)$ of the sequence satisfies $z^* = y^*$, and $z^*$ is a stationary point of \eqref{prob:feas}. Moreover, \eqref{cond1} holds.
\end{theorem}
\begin{proof}
  From the above discussion, the algorithm \eqref{scheme2} is just \eqref{scheme} as applied to \eqref{prob:feas}. Thus, in particular, one can pick $L = 1$ and $l = 0$ using properties of $\frac12 d_C^2$.
  In view of Theorem~\ref{thm1}, we only need to show that the sequence $\{(y^t,z^t,x^t)\}$ is bounded. We shall check that the conditions in Theorem~\ref{thm3.5} are satisfied.

  First, $f = \frac12 d_C^2$ and $g = \delta_D$ are clearly bounded from below.
  Now, if $D$ is compact, then $g = \delta_D$ is coercive. On the other hand, if $C$ is compact, then $f = \frac12 d_C^2$ is coercive. Consequently, Theorem~\ref{thm3.5} is applicable, from which we conclude that the sequence $\{(y^t,z^t,x^t)\}$ generated from \eqref{scheme2} is bounded.
\end{proof}

We have the following immediate corollary if $C$ and $D$ are closed semi-algebraic sets.

\begin{corollary}\label{cor1}
Let $C$ and $D$ be nonempty closed semi-algebraic sets, with $C$ being convex.  Suppose that $0<\gamma < \sqrt{\frac32} - 1$  and that either $C$ or $D$ is compact. Then the sequence $\{(y^t,z^t,x^t)\}$ converges to a point $(y^*,z^*,x^*)$ which satisfies $z^* = y^*$, and $z^*$ is a stationary point of \eqref{prob:feas}.
\end{corollary}
\begin{proof}
As $C$ is a semi-algebraic set, $y \mapsto \frac12 d_C^2(y)$ is a semi-algebraic function (see \cite[Lemma 2.3]{AtBoSv13}). Note that $D$ is also a semi-algebraic set, and so, $z \mapsto \delta_D(z)$ is also a semi-algebraic function. Thus, the conclusion follows from Theorem~\ref{thm2} and Theorem~\ref{thm1.5} with $f(y)=\frac12 d_C^2(y)$ and $g(z)=\delta_D(z)$.
\end{proof}
\begin{remark}\label{rem1}{\bf (Practical computation consideration on the step-size parameter)}
Though the upper bound on $\gamma$ given in \eqref{gammacond} might be too small in practice, it can be used in designing an update rule of $\gamma$ so that the resulting algorithm is guaranteed to converge (in the sense described by Theorem~\ref{thm2}). Indeed, similar to the discussion in \cite[Remark~2.1]{SunTohYang14}, one could initialize the algorithm with a large $\gamma$, and decrease the $\gamma$ by a constant ratio if $\gamma$ exceeds $\sqrt{\frac32} - 1$ and the iterate satisfies either $\|y^t - y^{t-1}\| > c_0/t$ for some prefixed $c_0>0$ or $\|y^t\| > c_1$ for some huge number $c_1>0$. In the worst case, one can obtain $0<\gamma < \sqrt{\frac32} - 1$ after finitely many decrease, and Theorem~\ref{thm2} shows that the sequence generated is bounded and clusters at stationary points when either $C$ or $D$ is compact. Otherwise, one must have $\|y^t - y^{t-1}\| \le c_0/t$ and $\|y^t\| \le c_1$ for all sufficiently large $t$. In this case, it follows from the first relation in \eqref{fyupdate} and the third relation in \eqref{scheme} that $\{x^t\}$ and $\{z^t\}$ are bounded. Moreover, we also see from \eqref{ineqbd}
that \eqref{cond1} holds. Thus, one can also show that the sequence generated is bounded and clusters at stationary points.
\end{remark}

In general, it is possible that the algorithm \eqref{scheme2} gets stuck at a stationary point that is not a global minimizer. Thus, there is no guarantee that this algorithm will solve the feasibility problem. However, a zero objective value of $d_C(y^*)$ certifies that $y^*$ is a solution of the feasibility problem, i.e., $y^*\in C\cap D$.

We next consider a specific case where $C = \{x\in \R^n:\; Ax = b\}$ for some matrix $A\in \R^{m\times n}$, $m \le n$, and $D$ is a closed semi-algebraic set. We show below that, if the $\{(y^t,z^t,x^t)\}$ generated by our DR splitting method converges to some $(y^*,z^*,x^*)$ with $z^*$ satisfying a certain constraint qualification, then the scheme indeed exhibits a local linear convergence rate.
To do this, we first prove an auxiliary lemma.
\begin{lemma}\label{lemma:bound}
Let $B \in \mathbb{R}^{p \times p}$ be a symmetric indefinite matrix. Then, there exists $\alpha>0$ such that for all $u \in \mathbb{R}^p$,
\[
 \|Bu\|^2 \ge \alpha \, (u^TBu).
\]
\end{lemma}
\begin{proof}
As $B$ is indefinite, $\{u: u^TBu=1\} \neq \emptyset$. Consider the following homogeneous quadratic optimization problem:
\begin{equation}\label{GTRS}
  \begin{array}{rl}
    \alpha = \inf\limits_{u \in \R^{p}} & \|Bu\|^2\\
    {\rm s.t.} & u^TBu = 1.
  \end{array}
\end{equation}
Clearly $\alpha \ge 0$. We now claim that $\alpha>0$. To see this, we proceed by the method of contradiction and suppose that there exists a sequence $\{u^t\}$ such that $(u^t)^TBu^t=1$ and $\|Bu^t\|^2 \rightarrow 0$.
Let $B=V^T\Sigma V$ be an eigenvalue decomposition of $B$, where $V$ is an orthogonal matrix and $\Sigma$ is a diagonal matrix. Letting $w^t=Vu^t$, we have
\[
({w^t})^T\Sigma w^t=1 \mbox{ and } ({w^t})^T\Sigma^2 w^t \rightarrow 0.
\]
Let $w^t=(w^t_{1},\ldots,w^t_{p})$ and $\Sigma={\rm Diag}(\lambda_1,\ldots,\lambda_{p})$. Then we see further that
\[
\sum_{i=1}^{p}\lambda_i(w^t_i)^2=1 \mbox{ and }  \sum_{i=1}^{p}\lambda_i^2(w^t_i)^2 \rightarrow 0.
\]
The second relation shows that either $\lambda_i=0$ or $w^t_i \rightarrow 0$ for each $i=1,\ldots,p$. This contradicts the first relation.
So, we must have $\alpha>0$, and hence the conclusion follows.
\end{proof}

\begin{proposition}{\bf (Local linear convergence rate under constraint qualification)}\label{prop:1}
Let $C = \{x \in \R^n:\; Ax = b\}$ and $D$ be a nonempty closed semi-algebraic set where $A \in \R^{m \times n}$, $m \le n$, and $b \in \R^m$.
Let $0 < \gamma < \sqrt{\frac32} - 1$ and suppose that the sequence $\{(y^t,z^t,x^t)\}$ generated from \eqref{scheme2} converges to $(y^*,z^*,x^*)$. Suppose,
in addition, that $C\cap D \neq \emptyset$ and the following constraint qualification holds:
\begin{equation}\label{CQ}
N_C\big(P_{C}(z^*)\big) \cap -N_D(z^*)=\{0\},
\end{equation}
where $N_S(a)$ is the (limiting) normal cone of $S$ at $a \in S$ defined in (\ref{eq:normal_cone}).
Then, $z^* \in C \cap D$ and there exist $\eta \in (0,1)$ and $\kappa>0$ such that for all large $t$,
\[
{\rm dist}\big(0, z^t-P_{C}(z^t)+N_D(z^t)\big) \le \kappa \, \eta^t.
\]
\end{proposition}
\begin{proof}
We first show that under the assumptions, we have $z^* \in C\cap D$ and $x^* = y^* = z^*$.
To this end, recall that any limit $(y^*,z^*,x^*)$ satisfies $y^*=z^* \in D$. From the optimality condition, we see also that
\[
0\in z^*-P_{C}(z^*)+N_D(z^*).
\]
On the other hand, note also that $z^*-P_{C}(z^*) \in N_C\big(P_{C}(z^*)\big)$. Hence, our assumption \eqref{CQ} implies that
$z^*-P_{C}(z^*)=0$ and so, $z^* \in C$. Thus, we have $y^*=z^* \in C\cap D$. Note that $y^* = \frac{1}{1+\gamma}(x^* + \gamma P_C(x^*))$, from which one can easily see that
$x^*=y^*$.

Before proceeding further, without loss of generality, we assume that $\D_\gamma(y^t,z^t, x^t) > \D_\gamma(y^*,z^*,x^*)$ and hence $(y^t,z^t, x^t) \neq (y^*,z^*,x^*)$ for all $t \ge 1$; since otherwise, the conclusions of the proposition follow easily.

Let $\D_\gamma(y,z,x) = \delta_D(z) + \hat \D_\gamma(y,z,x)$ where
\begin{eqnarray*}
  \hat \D_\gamma(y,z,x)&:=& \frac{1}{2}d_C^2(y)- \frac1{2\gamma}\|y - z\|^2 + \frac1\gamma \langle x-y,z-y\rangle \\
& = & \frac{1}{2}\|A^{\dag}(Ay-b)\|^2- \frac1{2\gamma}\|y - z\|^2 + \frac1\gamma \langle x-y,z-y\rangle.
\end{eqnarray*}
where $A^{\dag}$ is the pseudo inverse of the matrix $A$. Let us consider the function $h$ defined by
\[
h(y,z,x)=\hat \D_\gamma(y + y^*,z + z^*, x + x^*)-\hat \D_\gamma(y^*,z^*,x^*).
\]
Recall that $x^*=y^*=z^* \in C\cap D$. Hence, $h$ is a quadratic function with $h(0,0,0)=0$ and $\nabla h(0,0,0)=0$.
Thus, we have
$h(u)= \frac12 u^T B u$,
where $u = (y,z,x)$ and
\[
B=\nabla^2 h(0,0,0)=\left(\begin{array}{ccc}
A^{\dag}A+\frac{1}{\gamma}I_n & 0 & -\frac{1}{\gamma} I_n \\
0 & -\frac{1}{\gamma} I_n & \frac{1}{\gamma} I_n \\
-\frac{1}{\gamma} I_n & \frac{1}{\gamma} I_n & 0
\end{array} \right).
\]
Here, we use $I_n$ to denote the $n\times n$ identity matrix.
Clearly, $B$ is an indefinite $3n \times 3n$ matrix and so, the preceding lemma implies that there exists $\alpha>0$
such that $\|Bu\|^2 \ge \alpha \, u^TBu$ for all $u \in \mathbb{R}^{3n}$. Consequently, for any $u$ satisfying $u^TBu > 0$, we have
\begin{equation*}
  \|Bu\| \ge \sqrt{\alpha}\sqrt{u^TBu}.
\end{equation*}
Recall that $h(u)=\frac12 u^TBu$.  It then follows from the definition of $h$ that for all $t\ge 1$, we have
\begin{equation}\label{Dineqbd}
\begin{split}
\|\nabla \hat \D_\gamma(y^t,z^t, x^t)\| &\ge  \sqrt{2\alpha} \sqrt{\hat \D_\gamma(y^t,z^t, x^t)-\hat \D_\gamma(y^*,z^*,x^*)} \\
& =  \sqrt{2\alpha} \sqrt{\D_\gamma(y^t,z^t, x^t)- \D_\gamma(y^*,z^*,x^*)},
\end{split}
\end{equation}
where the equality follows from $\hat \D_\gamma(y^t,z^t, x^t)=\D_\gamma(y^t,z^t, x^t) > \D_\gamma(y^*,z^*,x^*)= \hat \D_\gamma(y^*,z^*,x^*)$ (thanks to $z^t \in D$).
Finally, to finish the proof, we only need to justify the existence of $\beta > 0$ such that for all large $t$,
\begin{equation}\label{eq:claim}
{\rm dist}(0,\partial \D_\gamma(y^t,z^t,x^t)) \ge \beta \|\nabla \hat \D_\gamma(y^t,z^t,x^t)\|.
\end{equation}
Then the conclusion of the proposition follows from Theorem~\ref{thm3} for the case $\theta=\frac12$ and Remark~\ref{rem:thm3}(i).

Note first that
\begin{equation*}
{\rm dist}(0,\partial \D_\gamma(y^t,z^t,x^t))= {\rm dist}\left(0,\{\nabla_y \hat \D_\gamma(y^t,z^t,x^t)\} \times \big(\nabla_z \hat \D_\gamma(y^t,z^t,x^t)+ N_D(z^t)\big) \times \{\nabla_x \hat \D_\gamma(y^t,z^t,x^t)\} \right),
\end{equation*}
To establish \eqref{eq:claim}, we only need to consider the partial subgradients with respect to $z$.
To this end, define $w^t := \nabla_z \hat \D_\gamma(y^t,z^t,x^t) = -\frac1\gamma (z^t - x^t)$ and let $v^t \in N_D(z^t)$ be such that
\[
{\rm dist}\left(0,\nabla_z \hat \D_\gamma(y^t,z^t,x^t)+ N_D(z^t)\right)= \left\|w^t+v^t\right\|.
\]

We now claim that, there exists $\theta \in [0,1)$  such that for all large $t$
\begin{equation}\label{eq:theta}
\langle w^t,v^t\rangle \ge -\theta \|w^t\| \cdot \|v^t\|.
\end{equation}
Otherwise, there exist $t_k \rightarrow \infty$ and $\theta_k \uparrow 1$ such that
\begin{equation}\label{eq:pp}
\langle w^{t_k},v^{t_k}\rangle < -\theta_k \|w^{t_k}\| \cdot \|v^{t_k}\|.
\end{equation}
In particular, $w^{t_k}\neq 0$ and $v^{t_k} \neq 0$. Furthermore, note that $v^t \in N_D(z^t)$ and
\[
w^t=-\frac{1}{\gamma}(z^t-x^t)=-\frac{1}{\gamma}(y^t-x^{t-1})=\frac{1}{1+\gamma}\left(x^{t-1}-P_C(x^{t-1})\right),
\]
where the second equality follows from the third relation in \eqref{scheme2} and the last relation follows from the first relation in \eqref{scheme2}.
By passing to a subsequence if necessary, we may assume that
\[
\frac{w^{t_k}}{\|w^{t_k}\|} \rightarrow w^* \in N_C(P_C(x^*)) \cap S =N_C(P_C(z^*)) \cap S \mbox{ and } \frac{v^{t_k}}{\|v^{t_k}\|} \rightarrow v^* \in N_D(z^*) \cap S,
\]
where $S$ is the unit sphere.
Dividing $\|w^{t_k}\| \|v^{t_k}\|$ on both sides of \eqref{eq:pp} and passing to the limit, we see that
$\langle w^*, v^* \rangle \le -1$.
This shows that $\|w^*+v^*\|^2=2+2\langle w^*, v^* \rangle \le 0$ and hence, $w^*=-v^*$. This contradicts \eqref{CQ} and thus \eqref{eq:theta} holds for some $\theta \in [0,1)$ and for all large $t$.

Now, using \eqref{eq:theta}, we see that for all large $t$
\begin{equation*}
\begin{split}
\left\|-\frac1\gamma (z^t - x^t)+v^t\right\|^2 &=\|w^t+v^t\|^2  =  \|w^t\|^2+\|v^t\|^2+2\langle w^t,v^t\rangle \\
 & \ge  \|w^t\|^2+\|v^t\|^2 - 2 \theta \|w^t\| \|v^t\|  \ge  (1-\theta) (\|w^t\|^2+\|v^t\|^2) \\
 & \ge (1-\theta) \left\|-\frac1\gamma (z^t - x^t)\right\|^2.
\end{split}
\end{equation*}
Therefore, for all large $t$
\begin{equation*}
\begin{split}
&{\rm dist}^2\left(0,\nabla_z \hat \D_\gamma(y^t,z^t,x^t)+ N_D(z^t)\right) =  \left\|-\frac1\gamma (z^t - x^t)+v^t\right\|^2 \\
& \ge  (1-\theta) \left\|-\frac1\gamma (z^t - x^t)\right\|^2 =(1-\theta) \|\nabla_z \hat \D_\gamma(y^t,z^t,x^t)\|^2.
\end{split}
\end{equation*}
Therefore, \eqref{eq:claim} holds with $\beta = \sqrt{1-\theta}$. Thus, the conclusion follows.
\end{proof}

\begin{remark}{\bf (Connection of our local convergence result to existing results)}
Assuming $C$ is affine and $D$ is super-regular, a similar local linear convergence result was
established in \cite[Theorem 3.18]{Luke1} for the classical DR splitting method considered in the literature for the feasibility problem, i.e., \eqref{scheme} applied to minimizing the sum of the indicator functions of the two sets $C$ and $D$. Our result is different from theirs in two aspects. First, we study the different algorithm \eqref{scheme2} which
is \eqref{scheme} applied to minimizing the squared distance function of $C$ subject to $D$.
Second, we look at semi-algebraic sets, and these sets are not necessarily super-regular in general. For a simple example, recall that the semi-algebraic set $D=\{(x_1,x_2):x_1x_2=0\}$ was shown in \cite[Remark 2.13]{Luke1} to be not super-regular.
\end{remark}

\begin{remark}
We also note that due to the nonconvex nature of the feasibility problem we consider, the local convergence requires a constraint qualification that depends on the limit
point $z^*$. Although the limit point $z^*$ is often hard to determine a priori, as we will see in Remark 7, this constraint qualification can be
regarded as an extension of the well-known linear regularity condition and is satisfied in many cases. Moreover, it is also possible that the constraint qualification is indeed
satisfied at every $z \in D$ for some pairs of sets $C$ and $D$. To see this, consider $C=\{(x_1,x_2):\;x_1=0\}$ and $D=\{(x_1,x_2):\; x_2 \ge -|x_1|\}$. Then we have
\begin{equation*}
  N_C(c) = \mathbb{R} \times \{0\},
\end{equation*}
for all $c=(c_1,c_2) \in C$, and for all $d=(d_1,d_2) \in D$
\begin{equation*}
  N_D(d) = \begin{cases}
    \{t(1,-1):\; t\ge 0\} & {\rm if}\ d_1 < 0, d_2=-|d_1|, \\
    \{t(-1,-1):\; t \ge 0\} & {\rm if}\ d_1 > 0, d_2=-|d_1|, \\
    \{t(1,-1):\; t\ge 0\} \cup \{t(-1,-1):\; t \ge 0\} & {\rm if}\ d_1 = 0, d_2=-|d_1|=0,\\
    \{(0,0)\} & {\rm if} \ d_2>-|d_2|.
  \end{cases}
\end{equation*}
Consequently, $N_C(P_C(z^*)) \cap -N_D(z^*)=\{(0,0)\}$ for all
$z^* \in D$.
\end{remark}

For a general nonconvex feasibility problem, i.e., to find a point in $\bigcap_{i=1}^M D_i$, with each $D_i$ being a nonempty closed set whose projection is easy to compute, it is classical to reformulate the problem as finding a point in the intersection of $H \cap (D_1\times D_2\times \cdots \times D_M)$, where
\begin{equation}\label{eq:H}
H = \{(x_1,\ldots,x_M):\; x_1 = \cdots = x_M\}.
\end{equation}
The algorithm \eqref{scheme2} can thus be applied. In addition, if it is known that $\bigcap_{i=1}^M D_i$ is bounded, one can further reformulate the problem as
finding a point in the intersection of $H_R\cap (D_1\times D_2\times \cdots \times D_N)$, where
\begin{equation}\label{eq:H_R}
H_R = \{(x_1,\ldots,x_M):\; x_1 = \cdots = x_M,\ \ \|x_1\|\le R\},
\end{equation}
and $R$ is an upper bound on the norms of the elements in $\bigcap_{i=1}^M D_i$.
We note that both the projections onto $H$ and $H_R$ can be easily computed.

We next state a corollary concerning the convergence of our DR splitting method as applied to finding a point in the intersection of $H \cap (D_1\times D_2\times \cdots \times D_M)$, assuming compactness of $D_i$, $i = 1,\ldots,M$. The proof is routine and is thus omitted.

\begin{corollary}{\bf (DR splitting method for general nonconvex feasibility problem)}
Let $D_1,\ldots,D_M$ be nonempty compact semi-algebraic sets in $\R^n$. Let $C=H$, where $H$ is defined as in \eqref{eq:H},
and let $D=D_1 \times \cdots \times D_M$. Let $0 < \gamma < \sqrt{\frac32} - 1$ and let the sequence $\{(y^t,z^t,x^t)\}$ be generated from \eqref{scheme2}. Then,
\begin{itemize}
\item[{\rm (i)}] the sequence $\{(y^t,z^t,x^t)\}$ converges to a point $(y^*,z^*,x^*)$, with $y^*=z^*$ and $z^*=(z_1^*,\ldots,z_M^*) \in D_1 \times \cdots D_M$ satisfying
\[
0 \in z^*_i-\frac{1}{M}\sum_{i=1}^M z^*_i + N_{D_i}(z^*_i), \ i=1,\ldots,M.
\]
\item[{\rm (ii)}]  Suppose, in addition, that $\bigcap_{i=1}^M D_i \neq \emptyset$ and the following constraint qualification holds:
\begin{equation}\label{CQ1}
 a_i \in N_{D_i}(z_i^*),\, i=1,\ldots,M,\  \sum_{i=1}^M a_i=0 \ \Rightarrow \ a_i=0, i=1,\ldots,M.
\end{equation}
Then, $z_1^*=\cdots=z_M^* \in \bigcap_{i=1}^M D_i$ and there exist $\eta \in (0,1)$ and $\kappa>0$ such that for all large $t$,
\begin{equation}\label{con2}
{\rm dist}\left(0, z^t_i-\frac{1}{M}\sum_{i=1}^M z^t_i + N_{D_i}(z^t_i)\right) \le \kappa \, \eta^t, \ i=1,\ldots,M.
\end{equation}
\end{itemize}
\end{corollary}
%
%
\begin{remark}
The constraint qualification \eqref{CQ1} is known as the linear regularity condition which is satisfied in many cases; for example, when $D$ is the Cartesian product of two transverse $C^2$-manifolds \cite[Theorem 5.2]{MOR}. It plays an important role in quantifying the
local linear convergence rate of the alternating projection method (see \cite[Theorem 4.3]{MOR} and \cite[Theorem 5.16]{Lewis}).
\end{remark}

Before closing this section, we use the example given in \cite[Remark~6]{Bauschke} to illustrate the difference in the behavior of our DR splitting method \eqref{scheme2}
and the classical DR splitting method considered in the literature for the feasibility problem, i.e., \eqref{scheme} applied to minimizing the sum of the indicator functions of the two sets.\footnote{We note that there are also another analytical examples constructed in \cite{Luke2} and \cite{Bauschke},
where the authors showed that the classical DR splitting method need not to be convergent. For
simplicity, we do not discuss it here.}
\begin{example}\label{example1}{\bf (Different behavior: our DR splitting method  vs the classical DR splitting method)}
  We consider $C = \{x\in \R^2:\; x_2 = 0\}$ and $D = \{(0,0),(7+\eta,\eta),(7,-\eta)\}$, where $\eta \in (0,1]$. It was discussed in \cite[Remark~6]{Bauschke} that
the DR splitting method, initialized at $x^0 = (7,\eta)$ and applied to minimizing the sum of indicator functions of the two sets, is not convergent; indeed, the sequence generated has a discrete limit cycle.
On the other hand, convergence of \eqref{scheme2} applied to this pair of sets is guaranteed by Corollary~\ref{cor1}, as long as $0<\gamma < \sqrt{\frac32}-1$. Below, we show explicitly that the generated sequence is convergent, and the
limit is $y^* = z^* = (7+\eta,\eta)$ and $x^*= (7+\eta,(1+\gamma)\eta)$.

To this end, we first consider a sequence $\{a_t\}$ defined by $a_1 = 2-\frac{1}{1+\gamma} > 0$ and for any $t\ge 1$,
\begin{equation}\label{recurrence}
a_{t+1} = \frac{\gamma}{1+\gamma}a_t + 1.
\end{equation}
Then $a_t > 0$ for all $t$ and we have
\[
a_{t+1} - a_t = \frac{\gamma}{1+\gamma}(a_t - a_{t-1}) = \cdots = \left(\frac{\gamma}{1+\gamma}\right)^{t-1}(a_2 - a_1).
\]
Consequently, $\{a_t\}$ is a Cauchy sequence and is thus convergent. Furthermore, it follows immediately from \eqref{recurrence} that $\lim_{t\rightarrow \infty}a_t = 1+\gamma$.

Now, we look at \eqref{scheme2} initialized at $x^0 = (7,\eta)$. Then $y^1 = \left(7,\frac{\eta}{1+\gamma}\right)$ and $2y^1 - x^0 = \left(7,\left[\frac{2}{1+\gamma}-1\right]\eta\right)$.
Since $\gamma < \sqrt{\frac{3}{2}}-1< \frac{3}{5}$, it is not hard to show that $z^1 = (7+\eta,\eta)$ and consequently $x^1 = \left(7+\eta,a_1\eta\right)$. Inductively, one can show that for all $t\ge 1$,
\[
y^{t+1} = \left(7+\eta,\frac{a_t}{1+\gamma}\eta\right),\ \ z^{t+1} = \left(7+\eta,\eta\right) \ \ {\rm and}\ \ x^{t+1} = \left(7+\eta,a_{t+1}\eta\right).
\]
Consequently, $y^* = z^* = (7+\eta,\eta)$ and $x^* = (7+\eta,(1+\gamma)\eta)$.
\end{example}

\section{Numerical simulations}\label{sec5}

In this section, we perform numerical experiments to test the DR splitting method on solving a nonconvex feasibility problem. All codes are written in MATLAB.

We consider the problem of finding an $r$-sparse solution of a linear system $Ax = b$.
To apply the DR splitting method, we let $C = \{x\in \R^n:\; Ax = b\}$ and $D = \{x\in \R^n:\; \|x\|_0 \le r,\ \|x\|_\infty \le 10^6\}$, where $\|x\|_0$ denotes the cardinality of $x$ and $\|x\|_\infty$ is the $\ell_\infty$ norm of $x$; the $10^6$ is just an arbitrary choice of large number to guarantee compactness of $D$. We benchmark our algorithm against the alternating projection method, which is an application of the proximal gradient algorithm with step-length $1$ to solve \eqref{prob:feas}. Specifically, in this latter algorithm, one initializes at an $x^0$ and updates
\begin{equation*}
  x^{t+1} \in \Argmin_{\|x\|_0\le r, \|x\|_\infty\le 10^6} \left\{\|x - (x^t + A^\dagger(b - Ax^t))\|\right\};
\end{equation*}
a closed-form solution for this subproblem can be found in \cite[Proposition~3.1]{Lu_Zhang13}.
We initialize both algorithms at the origin and terminate them when
\[
\frac{\max\{\|x^t - x^{t-1}\|,\|y^t - y^{t-1}\|,\|z^t - z^{t-1}\|\}}{\max\{\|x^{t-1}\|,\|y^{t-1}\|,\|z^{t-1}\|,1\}}< 10^{-8}
\ \ {\rm and}\ \ \frac{\|x^t - x^{t-1}\|}{\max\{\|x^{t-1}\|,1\}}< 10^{-8}
\]
respectively, for the DR splitting method and the alternating projection method.
Furthermore, for the DR splitting method, we adapt the heuristics described in Remark~\ref{rem1}: we initialize $\gamma = 150\cdot\gamma_0$ and update $\gamma$ as
$\max\{\frac\gamma 2,0.9999\cdot\gamma_0\}$ whenever $\gamma > \gamma_0 := \sqrt{\frac32} - 1$, and the sequence satisfies either $\|y^t-y^{t-1}\| > \frac{1000}t$ or $\|y^t\| > 10^{10}$.\footnote{We also solved a couple instances using directly a small $\gamma < \gamma_0$ in the DR splitting method. The sequence generated tends to get stuck at stationary points that are not global minimizers.}

We generate random linear systems with sparse solutions. We first generate an $m\times n$ matrix $A$ with i.i.d. standard Gaussian entries. We then randomly generate an $\hat x\in \R^r$ with $r = \lceil \frac m5\rceil$, again with i.i.d. standard Gaussian entries. A random sparse vector $\tilde x\in \R^n$ is then generated by first setting $\tilde x = 0$ and then specifying $r$ random entries in $\tilde x$ to be $\hat x$. Finally, we set $b = A\tilde x$.

In our experiments, for each $m=100$, $200$, $300$, $400$ and $500$, and $n=4000$, $5000$ and $6000$, we generate $50$ random instances as described above. The computational results are reported in Table~\ref{t1}, where we report the number of iterations averaged over the $50$ instances, as well as the maximum and minimum function values at termination (fval$_{\max}$ and fval$_{\min}$).\footnote{We report $\frac12 d_C^2(z^t)$ for DR splitting, and $\frac12 d_C^2(x^t)$ for alternating projection.} We also report the number of successes and failures (succ and fail), where we declare a success if the function value at termination is below $10^{-12}$, and a failure if the value is above $10^{-6}$.\footnote{The two thresholds are different and hence ${\rm succ} + {\rm fail}$ is not necessarily $50$. We set different thresholds so as to see if it is easy to determine whether the method has got stuck at stationary points that are not global minimizers.} We observe that both methods fail more often for harder instances (smaller $m$), and the DR splitting method clearly outperforms the alternating projection method in terms of both the number of iterations and the solution quality.

\begin{table}[h]
\small
\begin{center}
\begin{tabular}{|r|r||r|r|r|r|r||r|r|r|r|r|}  \hline
\multicolumn{2}{|c||}{Data} & \multicolumn{5}{c||}{DR} & \multicolumn{5}{c|}{Alt Proj}
\\ 
$m$ & $n$ & ${\rm iter}$ & ${\rm fval}_{\max}$ & ${\rm fval}_{\min}$ & succ &  fail & ${\rm iter}$ & ${\rm fval}_{\max}$ & ${\rm fval}_{\min}$ & succ &  fail
\\ \hline
   100 &  4000 &    1967 & 3e-02 & 6e-17 & 30 & 20 &   1694 & 8e-02 & 4e-03 &  0 & 50 \\
   100 &  5000 &    2599 & 2e-02 & 2e-16 & 18 & 32 &   1978 & 7e-02 & 5e-03 &  0 & 50 \\
   100 &  6000 &    2046 & 1e-02 & 1e-16 & 12 & 38 &   2350 & 5e-02 & 4e-05 &  0 & 50 \\
   200 &  4000 &     836 & 2e-15 & 2e-16 & 50 &  0 &   1076 & 3e-01 & 3e-05 &  0 & 50 \\
   200 &  5000 &    1080 & 3e-15 & 2e-16 & 50 &  0 &   1223 & 2e-01 & 2e-03 &  0 & 50 \\
   200 &  6000 &    1279 & 7e-02 & 1e-16 & 43 &  7 &   1510 & 2e-01 & 1e-13 &  1 & 49 \\
   300 &  4000 &     600 & 3e-15 & 2e-16 & 50 &  0 &    872 & 4e-01 & 6e-14 &  3 & 46 \\
   300 &  5000 &     710 & 4e-15 & 4e-16 & 50 &  0 &   1068 & 3e-01 & 9e-14 &  3 & 45 \\
   300 &  6000 &     812 & 3e-15 & 2e-16 & 50 &  0 &   1252 & 3e-01 & 1e-13 &  1 & 49 \\
   400 &  4000 &     520 & 2e-15 & 3e-17 & 50 &  0 &    818 & 6e-01 & 7e-14 & 30 & 19 \\
   400 &  5000 &     579 & 3e-15 & 5e-16 & 50 &  0 &    946 & 4e-01 & 9e-14 & 12 & 36 \\
   400 &  6000 &     646 & 4e-15 & 6e-16 & 50 &  0 &   1108 & 3e-01 & 1e-13 &  4 & 44 \\
   500 &  4000 &     499 & 1e-16 & 1e-18 & 50 &  0 &    640 & 4e-01 & 6e-14 & 38 & 10 \\
   500 &  5000 &     519 & 1e-15 & 4e-17 & 50 &  0 &    846 & 4e-01 & 8e-14 & 37 & 13 \\
   500 &  6000 &     556 & 3e-15 & 3e-16 & 50 &  0 &   1071 & 5e-01 & 1e-13 & 22 & 28 \\
\hline
\end{tabular}
\end{center}
\caption{Comparing Douglas-Rachford splitting and alternating projection on random instances.}\label{t1} \normalsize
\end{table}

Finally, as suggested by one of the reviewers, we also consider the classical DR splitting method, i.e., the DR splitting method applied to minimizing the sum of indicator functions of the sets $C$ and $D$. As discussed in Example~\ref{example1}, this method can be non-convergent in general.

In our numerical tests, we use the same initialization and termination criteria as our DR splitting method. In addition, we also terminate the algorithm once the number of iterations exceeds $20000$. We solve exactly the same $50$ instances for each $m=100$, $200$, $300$, $400$ and $500$, and $n=4000$, $5000$ and $6000$ from Table~\ref{t1}. The computational results are reported in Table~\ref{t2}, where, as before, we report the number of iterations averaged over the $50$ instances, the maximum and minimum function values $\frac12 d_C^2(z^t)$ at termination, and the number of successes and failures defined as above. One can observe that this approach is slower than both our DR splitting method and the alternating projection method, while its solution quality is worse than our DR splitting method, but is better than the alternating projection method.

\begin{table}[h]
\small
\begin{center}
\begin{tabular}{|r|r||r|r|r|r|r|}  \hline
\multicolumn{2}{|c||}{Data} & \multicolumn{5}{c|}{DR applied to ${\delta_C + \delta_D}$}
\\ 
$m$ & $n$ & ${\rm iter}$ & ${\rm fval}_{\max}$ & ${\rm fval}_{\min}$ & succ &  fail
\\ \hline
   100 &  4000 &   20000 & 2e+00 & 3e-03 &  0 & 50 \\
   100 &  5000 &   19821 & 2e+00 & 1e-16 &  1 & 49 \\
   100 &  6000 &   20000 & 1e+00 & 1e-04 &  0 & 50 \\
   200 &  4000 &   10595 & 2e+00 & 1e-16 & 33 & 16 \\
   200 &  5000 &   16251 & 2e+00 & 6e-18 & 19 & 31 \\
   200 &  6000 &   17636 & 3e+00 & 2e-16 & 11 & 39 \\
   300 &  4000 &    4649 & 2e+00 & 9e-17 & 45 &  5 \\
   300 &  5000 &    6926 & 4e+00 & 2e-16 & 43 &  7 \\
   300 &  6000 &   12236 & 4e+00 & 6e-17 & 25 & 25 \\
   400 &  4000 &    2322 & 4e-15 & 2e-16 & 50 &  0 \\
   400 &  5000 &    3357 & 2e+00 & 3e-16 & 49 &  1 \\
   400 &  6000 &    4487 & 2e+00 & 2e-16 & 48 &  2 \\
   500 &  4000 &    1986 & 6e-15 & 4e-16 & 50 &  0 \\
   500 &  5000 &    2859 & 5e-15 & 3e-16 & 50 &  0 \\
   500 &  6000 &    3065 & 5e-15 & 1e-16 & 50 &  0 \\
\hline
\end{tabular}
\end{center}
\caption{Computational results for the DR splitting applied to minimizing $\delta_C+\delta_D$ on the same random instances from Table~\ref{t1}.}\label{t2} \normalsize
\end{table}

\section{Concluding remarks}\label{sec6}
In this paper, we examine the convergence behavior of the Douglas-Rachford splitting method when applied to solving nonconvex optimization problems,
particularly, the nonconvex feasibility problem.
By introducing the Douglas-Rachford merit function, we prove the global convergence and establish local convergence rates for the DR splitting method when the step-size
parameter $\gamma$ is chosen sufficiently small (with an explicit threshold) and the sequence generated is bounded. We also provide simple sufficient conditions that guarantee the boundedness of the sequence generated from the DR splitting method. Preliminary numerical experiments are performed, which indicate that the DR splitting method usually outperforms the alternating projection method in finding a sparse solution of a linear system, in terms of both solution quality and number of iterations taken.\\

\noindent {\bf Acknowledgement.} We would like to thank the two anonymous referees for their comments that helped improve the manuscript.

\end{document}